\documentclass{amsart}
\usepackage[utf8]{inputenc}

\usepackage[top=1in, bottom=1in, left=1.25in, right=1.25in]{geometry}

\usepackage{amsmath}
\usepackage{amsthm}

\makeatletter
\newtheorem*{rep@theorem}{\rep@title}
\newcommand{\newreptheorem}[2]{%
	\newenvironment{rep#1}[1]{%
		\def\rep@title{#2 \ref{##1}}%
		\begin{rep@theorem}}%
		{\end{rep@theorem}}}
\makeatother

\newtheorem{lemma}{Lemma}[section]
\newtheorem{theorem}{Theorem}[section]
\newreptheorem{theorem}{Theorem}
\newreptheorem{lemma}{Lemma}
\newtheorem{definition}{Definition}[section]
\newtheorem{proposition}{Proposition}[section]
\newtheorem{corollary}{Corollary}[section]

\usepackage{mathrsfs}
\usepackage{amsfonts,amssymb}
\usepackage{romannum}
\usepackage{bbm}

\usepackage{enumitem}

\usepackage{hyperref}

\usepackage{tikz}

\usepackage{biblatex}
\def\p{\partial}

\title{A Framework for Gluing Harmonic Maps}
\author{Shaozong Wang}
\date{\today}

\begin{document}
\begin{sloppypar}
		
\begin{abstract}
	In this paper, we study the gluing construction of the extended harmonic maps between Riemannian manifolds. Harmonic maps are critical points of the energy functional. We construct the gluing map of the extended harmonic maps from Riemann surfaces to some Riemannian manifold $N$ under certain conditions.
\end{abstract}

\pagenumbering{arabic}
		
\maketitle
		
\tableofcontents
		
\section{Introduction}

\subsection{Harmonic Maps}

This subsection is taken from Lin and Wang \cite[Section 1.1, p. 1]{lin_wang}. Let $f: (M, g) \rightarrow (N, h)$ be a smooth map between smooth Riemannian manifolds. For any fixed $p \in M$, there exist two normal coordinate charts $U_p \subset M$ of $p$ and $V_q \subset N$ of $q = f(p)$ such that $f(U_p) \subset V_q$. The Dirichlet energy density function $e(f)$ is defined by
\begin{equation*}
	e(f) (x) \left( \equiv \left|\nabla f\right|_g^2 \right) = \frac{1}{2} \sum_{\alpha, \beta} g^{\alpha\beta} (x) h_{ij} (f(x)) \frac{\p f^i}{\p x^\alpha} \frac{\p f^j}{\p x^\beta},
\end{equation*}
where $\left(x^\alpha\right)$ and $\left(f^i\right)$ are the coordinate systems on $U_p$ and $V_q$ respectively. The Dirichlet energy of $f$ is defined as
\begin{equation*}
	E(f) := \int_M e(f) dv_g,
\end{equation*}
and we have the following definition and proposition:
\begin{definition}
	A map $f \in C^2 (M, N)$ is a harmonic map, if it is a critical point of the Dirichlet energy functional $E$.
\end{definition}

\begin{proposition}
	\label{ELEquation}
	A map $f \in C^2 (M, N)$ is a harmonic map iff $f$ satisfies
	\begin{equation*}
		g^{\alpha\beta} \left( f^k_{\alpha\beta} - (\Gamma^M)^\gamma_{\alpha\beta} f^k_\gamma + (\Gamma^N)^k_{ij} (f) f^i_\alpha f^j_\beta \right) \frac{\p}{\p y^k}= 0,
	\end{equation*}
	on $M$, where we denote $\frac{\p f}{\p x^\alpha}$ by $f_\alpha$, and $\Gamma^M$, $\Gamma^N$ are the Christoffel symbols of the metric on $M$ and $N$, respectively.
\end{proposition}

\subsection{Related Works}

There is extensive work on gluing in context of connections, metrics, and pseudoholomorphic curves. Taubes \cite{TauSelfDual} \cite{taubes1983} \cite{TauIndef} \cite{TauFrame} discussed gluing for anti-self-dual (ASD) connections and Yang-Mills connections. Gluing for ASD connections has also been explored by Donaldson \cite{DonConn}, Mrowka \cite{MrowkaThesis}, Feehan and Leness \cite{FLKM1}. Gluing for Seiberg-Witten monopoles is discussed by Fr{\o}yshov \cite{Froyshov_2008} and G.J. Parker \cite{parker2024gluingmathbbz2harmonicspinors}, and for Non-Abelian monopoles by Feehan and Leness \cite{FL3}. Brendle \cite{brendle2003constructionsolutionsyangmillsequations} discussed gluing for Yang-Mills connections. Brendle and Kapouleas \cite{BK} studied the gluing method for Eguchi-Hanson metrics. Kapouleas \cite{Kapouleas} studied gluing for minimal immersions. Breiner, Kapouleas, and Kleene \cite{BreinerKapouleasKleene} studied gluing for constant mean curvature surfaces. Additionally, gluing in the context of pseudoholomorphic curves is covered in  Fukaya \cite{Fukaya_Oh_Ohta_Ono_kuranishi_structures_virtual_fundamental_chains}, Hutchings and Taubes \cite{taubes_obstruction} \cite{taubes_obstruction2}, Abouzaid \cite{abouzaid}, McDuff and Salamon \cite{McDuffSalamon2}, as well as McDuff and Wehrheim \cite{McDuff_Wehrheim_2017gt} \cite{McDuff_Wehrheim_2017plms}. Malchiodi, Rupflin, and Sharp \cite{malchiodi2023lojasiewiczinequalitiesnearsimple} and Rupflin \cite{rupflin2022lojasiewiczinequalitiesharmonicmaps} considered gluing for almost-harmonic maps. Chen and Tian \cite{ChenTian} studied energy estimates for harmonic maps.

\subsection{Main Results}

In this thesis we consider two harmonic maps, $f_1: \Sigma_1 \rightarrow N$ and $f_2: \Sigma_2 \rightarrow N$, where $\Sigma_1$ and $\Sigma_2$ are Riemann surfaces and $N$ is a closed Riemannian manifold. We consider under what conditions can we glue these two maps and, when the gluing map exists, what properties does the gluing map possess.

\subsection{Pregluing of Manifolds and Maps}
\label{gluing_set_up}

We are given two Riemann surfaces, $\Sigma_1$ and $\Sigma_2$, a closed Riemannian manifold $N$, and two harmonic maps $f_1: \Sigma_1 \rightarrow N$ and $f_2: \Sigma_2 \rightarrow N$. Suppose $f_1 (x_1) = f_2 (x_2)$, then we connect $\Sigma_1$ and $\Sigma_2$ by punching holes at $x_1$ and $x_2$, then gluing them by a neck, which are annuli $A_i(\delta, R)$ centered at $x_i$ with radius $r$,  $\delta /2R < r < 2/\delta R$, $i = 1, 2$.  Denote this by $\Sigma_1 \#_{\delta, R} \Sigma_2$, where $\delta$ and $R$ are parameters of the neck. We use cutoff functions to piece together the two maps. See Section \ref{pregluing_detail} for details.

\subsection{Existence of the Extended Gluing Map}

We assume that we are a Riemann surface $\Sigma$ and a closed Riemannian manifold $N$. Consider a smooth map $f: \Sigma \rightarrow N$. For $\xi \in f^{-1} TN$, we can consider the perturbation of $f$ by $\xi$ under the exponential map, which we write as $\exp_f (\xi)$. We define $W^{2, p} (\Sigma, f^{-1} TN)$ to be the Sobolev space consisting all $\xi$ that are $W^{2, p}$ in each coordinate chart, with the usual Sobolev space structure. We can define $L^p (\Sigma, f^{-1} TN)$ spaces similarly.

We construct a section of $\left(\exp_f (\xi)\right)^{-1} TN$ where each component is obtained by plugging the perturbation into the equations in Proposition \ref{ELEquation}:
\begin{equation*}
	g^{\alpha\beta} \left( (\exp_f (\xi))^k_{\alpha\beta} - (\Gamma^M)^\gamma_{\alpha\beta} (\exp_f (\xi))^k_\gamma + (\Gamma^N)^k_{ij} (\exp_f (\xi)) (\exp_f (\xi))^i_\alpha (\exp_f (\xi))^j_\beta \right) \frac{\p}{\p y^k}.
\end{equation*}
We get a section of $f^{-1} TN$ when we compose the above with parallel transport from $\exp_f (\xi)$ to $f$. Thus we define the harmonic map operator $\mathcal{F}$ from the space $W^{2, p}_{f}$ to $L^p_{f}$ as
\begin{align*}
	\mathcal{F}_f(\xi) &:= \Phi_f(\xi)^{-1}\left(g^{\alpha\beta} \left( (\exp_f (\xi))^k_{\alpha\beta} - (\Gamma^M)^\gamma_{\alpha\beta} (\exp_f (\xi))^k_\gamma +\right.\right. \\
	&\left. \left. (\Gamma^N)^k_{ij} (\exp_f (\xi)) (\exp_f (\xi))^i_\alpha (\exp_f (\xi))^j_\beta \right) \frac{\p}{\p y^k}\right).
\end{align*}
where $g$ denotes the metric on $\Sigma$, and $\Phi$ is the parallel transport from $f$ to $\exp_f (\xi)$ along the geodesic $\exp_f (t\xi)$. See Section \ref{operator_def} for details.

Now we define the domain of our gluing map, which is a space of pairs of harmonic maps satisfying certain conditions. These conditions will allow us to carry out a construction similar to that in McDuff and Salamon \cite[Chapter 10]{mcduff}. The following definition is an analog of the definition of $\mathcal{M} (c)$ in McDuff and Salamon \cite[Section 10.1]{mcduff}.

\begin{definition}[Space of Harmonic Map Pairs] (Compare McDuff and Salmon \cite[Section 10.1, p. 371]{mcduff} for the analogous Definition for J-holomorphic curves)
	\label{moduli_defn_nonsurj}
	Fix domain manifolds (closed, of dimension 2, with Riemannian metric) $\Sigma_1$ and $\Sigma_2$ and closed Riemannian manifold $N$. Fix a constant $1 < p < 2$. Fix points $x_1 \in \Sigma_1$ and $x_2 \in \Sigma_2$. Let $\mathcal{M} (c, p)$ denote the set of all pairs of harmonic maps $(f_1, f_2)$ such that
	\begin{enumerate}
		\item $y := f_1 (x_1) = f_2 (x_2)$.
		\item $\|df_i\|_{L^\infty} \leq c$ and $\|d^2 f_i\|_{L^\infty} \leq c$ for $i = 1, 2$.
	\end{enumerate}
\end{definition}

In this thesis, we will need to specify the parameters of the neck in a connected sum.

\begin{definition}[Set of Parameter Pairs] (Compare McDuff and Salmon \cite[Section 10.1, p. 371]{mcduff} for the analogous definition for J-holomorphic curves)
	For any $0 < \delta_0 < 1$, we define $\mathcal{A} (\delta_0)$ to be the set of all pairs of $(\delta, R)$ such that $0 < \delta < \delta_0$ and $\delta R > 1 / \delta_0$.
\end{definition}

We will also need the following spaces and operators in the construction of the approximate inverse.
\begin{definition}(Compare McDuff and Salmon \cite[Section 10.5, p. 382]{mcduff} for the analogous Definition for J-holomorphic curves)
	\label{spaces_and_operator}
	Define the spaces
	\begin{equation*}
		W^{2,p}_f := W^{2,p} (\Sigma, f^{-1}TN),\quad L^p_f := L^p (\Sigma, f^{-1} TN)
	\end{equation*}
	for $f : \Sigma \rightarrow N$. Given $f_1: \Sigma_1 \rightarrow N, f_2 : \Sigma_2 \rightarrow N$ such that $f_1(x_1) = f_2 (x_2)$ for some $x_1 \in \Sigma_1$ and $x_2 \in \Sigma_2$, denote
	\begin{equation}
		\label{W2p}
		W^{2,p}_{f_{1, 2}} := \left\{ (\xi_1 , \xi_2) \in W^{2,p}_{f_1} \times W^{2,p}_{f_2} \mid \xi_1 (x_1) = \xi_2 (x_2)\right\}.
	\end{equation}
	We define $D_{1, 2} : W^{2,p}_{f_{1, 2}} \rightarrow L^p_{f_1} \times L^p_{f_2}$ by setting
	\begin{equation}
		\label{D0inf}
		D_{f_1, f_2} (\xi_1, \xi_2) := (D_{f_1} \xi_1 , D_{f_2} \xi_2 ).
	\end{equation}
\end{definition}
Note that for the above definition, when there is no ambiguity, we can use $D_{1, 2}$. Otherwise, we will write out the maps.

Consider $(\tilde{f}_1, \tilde{f}_2) \in \mathcal{M} (c, p)$. We will show in Lemma \ref{fredholm} that $D_{\tilde{f}_i}$ is a Fredholm operator for $i = 1, 2$. Then we can show that $D_{\tilde{f}_1, \tilde{f}_2}$ is Fredholm (see Lemma \ref{D12fredholm}), and that we can choose representatives of the cokernel that are supported away from some open neighborhoods of $x_1$ and $x_2$ (see Lemma \ref{choice2}). Thus we can make the following definition.
\begin{definition}
	\label{V}
	Consider the cokernel of $D_{\tilde{f}_1, \tilde{f}_2}$ defined as the quotient space\\
	$(L^p (\Sigma_1, \tilde{f}_1^{-1} TN) \times L^p (\Sigma_2, \tilde{f}_2^{-1} TN)) / \text{Im} D_{\tilde{f}_1, \tilde{f}_2}$. Choose linearly independent representatives of the quotient spaces such that they
	\begin{enumerate}
		\item are supported away from some open neighborhoods of $x_1$ and $x_2$,
		\item span the cokernel of $D_{\tilde{f}_1, \tilde{f}_2}$.
	\end{enumerate}
	Suppose that the set of such elements is $\{\tilde{v}_1, \cdots, \tilde{v}_k\}$. Let $\tilde{V}$ be the space spanned by $\{\tilde{v}_1, \cdots, \tilde{v}_k\}$.
\end{definition}

\begin{theorem}[Existence of the Extended Gluing Map]
	\label{existence_thm}
	For any $(\tilde{f}_1, \tilde{f}_2) \in \mathcal{M} (c, p)$, there exists a neighborhood $\mathcal{U}$ in $\mathcal{M} (c, p)$ and $\delta_0 = \delta_0 (c, p, \Sigma_1,  \Sigma_2, x_1, x_2, N, \mathcal{U})> 0$, such that for each pair of $(\delta, R) \in \mathcal{A} (\delta_0)$, there exists a gluing map $\imath_{\delta, R}: \mathcal{U} \rightarrow W^{2, p} (\Sigma_1 \#_{\delta, R} \Sigma_2, N) \times \tilde{V}$ such that each element $(\exp_{f^R} \xi, \tilde{v}) \in \imath_{\delta, R} (\mathcal{U})$ satisfies
	\begin{equation}
		\label{extended_harmonic_map_equation}
		\mathcal{F}_{f^R} (\xi) + v = 0
	\end{equation}
	where $\Sigma_1 \#_{\delta, R} \Sigma_2$ is the glued manifold as defined in \ref{weighted_metric}, and $\tilde{V}$ is defined in Definition \ref{V}, and $v = \sigma (\tilde{v})$, where $\sigma$ is defined in Definition \ref{sigma}. Furthermore, for any $\epsilon > 0$, we can choose $\delta_0 = \delta_0 (c, p, \Sigma_1,  \Sigma_2, x_1, x_2, N, \mathcal{U})$ such that, for any $(f_1, f_2) \in \mathcal{U}$, there exists $\xi \in W^{2, p} (\Sigma_1 \#_{\delta, R} \Sigma_2)$ satisfying
	\begin{equation*}
		\imath_{\delta, R} ((f_1, f_2)) = (\exp_{f^R} \xi, \tilde{v}), \quad \|(\xi^R, \tilde{v})\|_{2, p, R, V} < \epsilon
	\end{equation*}
	where $f^R$ denotes the pregluing of $f_1, f_2$ defined in (\ref{fR}), and the norm is defined in Definition \ref{weighted_with_V}.
	
	In particular, consider 
	\begin{equation*}
		\imath_{\delta, R} ((f_1, f_2)) \mid_{\tilde{V}} = 0.
	\end{equation*}
	If there are elements in $\imath_{\delta, R} (\mathcal{U})$ that satisfy the above equation, then these elements form a subset of the image of the gluing map consisting of harmonic maps. Otherwise, there is no harmonic map in the image of the gluing map.
\end{theorem}
Note that the last paragraph of the above Theorem reduces the infinite dimensional problem of solving a harmonic map equation on a Banach space to a finite dimensional one, which is setting
\begin{equation*}
	v = 0
\end{equation*}
in (\ref{extended_harmonic_map_equation}) and solving for $\xi$. This is similar to the ideas in Hutchings and Taubes \cite{taubes_obstruction} \cite{taubes_obstruction2} and Taubes \cite{taubes_gauge} \cite{taubes_pseudo}.

In particular, we get the following corollary by letting $\Sigma_2 = S^2$.
\begin{corollary}
	\label{existence_coro}
	Let $\Sigma_2 = S^2$. For any $(\tilde{f}_1, \tilde{f}_2) \in \mathcal{M} (c, p)$, there exists a neighborhood $\mathcal{U}$ in $\mathcal{M} (c, p)$ and $\delta_0 = \delta_0 (c, p, \Sigma_1, x_1, x_2, N, \mathcal{U})> 0$, such that for each pair of $(\delta, R) \in \mathcal{A} (\delta_0)$, there exists a gluing map $\imath_{\delta, R}: \mathcal{U} \rightarrow W^{2, p} (\Sigma_1 \#_{\delta, R} S^2, N) \times \tilde{V}$ such that each element $(\exp_{f^R} \xi, \tilde{v}) \in \imath_{\delta, R} (\mathcal{U})$ satisfies
	\begin{equation}
		\label{extended_harmonic_map_equation}
		\mathcal{F}_{f^R} (\xi) + v = 0
	\end{equation}
	where $\Sigma_1 \#_{\delta, R} S^2$ is the glued manifold as defined in \ref{weighted_metric}, and $\tilde{V}$ is defined in Definition \ref{V}, and $v = \sigma (\tilde{v})$, where $\sigma$ is defined in Definition \ref{sigma}. Furthermore, for any $\epsilon > 0$, we can choose $\delta_0 = \delta_0 (c, p, \Sigma_1, x_1, x_2, N, \mathcal{U})$ such that, for any $(f_1, f_2) \in \mathcal{U}$, there exists $\xi \in W^{2, p} (\Sigma_1 \#_{\delta, R} S^2)$ satisfying
	\begin{equation*}
		\imath_{\delta, R} ((f_1, f_2)) = (\exp_{f^R} \xi, \tilde{v}), \quad \|(\xi^R, \tilde{v})\|_{2, p, R, V} < \epsilon
	\end{equation*}
	where $f^R$ denotes the pregluing of $f_1, f_2$ defined in (\ref{fR}), and the norm is defined in Definition \ref{weighted_with_V}.
	
	In particular, consider 
	\begin{equation*}
		\imath_{\delta, R} ((f_1, f_2)) \mid_{\tilde{V}} = 0.
	\end{equation*}
	If there are elements in $\imath_{\delta, R} (\mathcal{U})$ that satisfy the above equation, then these elements form a subset of the image of the gluing map consisting of harmonic maps. Otherwise, there is no harmonic map in the image of the gluing map.
\end{corollary}

\subsection{Outline}
In Section \ref{basic_setup}, we introduce the setup of the problem. We define the harmonic map operator and compute the linearization. Then we use the same idea as in Donaldson and Kronheimer \cite[Proposition 7.2.28]{donaldson_kronheimer} to construct a surjective operator. In Section \ref{existence_of_the_gluing_map}, we write out the details of the construction of the gluing map. We introduce the Implicit Function Theorem used in the proof, and the pregluing construction. Then we construct the approximate inverse and from that construct the real right inverse. Finally, we check the conditions needed in the Implicit Function Theorem and show the existence of the gluing map. The appendix contains technical details such as norm and derivative estimates, and related facts in elliptic PDE theory and functional analysis.

\subsection{Acknowledgement}

The author expresses deep gratitude to Professor Paul Feehan for his invaluable guidance and unwavering support, and to Professor Dan Ketover for his insightful comments and suggestions. Special thanks are also extended to Professor Jason Lotay, Professor Thomas H. Parker, Zilu Ma, Gregory Parker, Junsheng Zhang, Xiao Ma, Liuwei Gong, and Jiakai Li for their helpful discussions. This work is also based in part on research supported by the National Science Foundation under Grant No. 1440140, while the author was in residence at the Simons Laufer Mathematical Sciences Institute in Berkeley, California, during Fall 2022 as an associate of the program Analytic and Geometric Aspects of Gauge Theory.

\section{Basic Setup}
\label{basic_setup}

\subsection{Pregluing of Manifolds and Maps}
\label{pregluing_detail}

The following setup and notation are similar to those in chapter 10 of \cite{mcduff}.

Consider gluing $f^0$ and $f^\infty$ whose domains are both $S^2$ and codomains $N$. We denote
\begin{equation*}
	y := f^0 (0) = f^\infty (\infty).
\end{equation*}

\begin{center}
	\begin{tikzpicture}
		
		\def\shiftX{-5}
		
		\draw[thick] (\shiftX,0) circle (2cm);
		\draw[thick] (\shiftX+2,0) arc (0:180:2cm and 0.5cm);
		\draw[dashed] (\shiftX-2,0) arc (180:360:2cm and 0.5cm);
		\draw[thick] (\shiftX,-2) arc (-90:90:0.5cm and 2cm);
		\draw[dashed] (\shiftX,2) arc (90:270:0.5cm and 2cm);
		
		\node[above] at (\shiftX,2) {\(\infty\)};
		\node[below] at (\shiftX,-2) {\(0\)};
		
		\draw[thick] (\shiftX,6) circle (2cm);
		\draw[thick] (\shiftX+2,6) arc (0:180:2cm and 0.5cm);
		\draw[dashed] (\shiftX-2,6) arc (180:360:2cm and 0.5cm);
		\draw[thick] (\shiftX,4) arc (-90:90:0.5cm and 2cm);
		\draw[dashed] (\shiftX,8) arc (90:270:0.5cm and 2cm);
		
		\node[above] at (\shiftX,8) {\(\infty\)};
		\node[below] at (\shiftX,4) {\(0\)};
		
		\draw[thick] plot [smooth cycle] coordinates {(3,3) (4,5) (6,6) (7,4) (5,2)};
		\node at (5.2,4) {\(N\)};
		
		\draw[->, thick] (\shiftX+2,1) to [bend left=20] (3,4);
		\node[above] at (0.2,2.5) {\( f^\infty \)};
		
		\draw[->, thick] (\shiftX+2,6.5) to [bend left=20] (3,5);
		\node[above] at (0.2, 5.5) {\( f^0 \)};
		
	\end{tikzpicture}
\end{center}

Since $S^2$ is compact, we know there exists $c>0$ such that $\|df^0\|_{L^{\infty}}\leq c$ and $\|df^\infty\|_{L^{\infty}}\leq c$ (Here the norms are in the sense of the round metric).

Let $\epsilon$ be less than the injective radius of $(N, h)$. By the upper bound of the differential of the harmonic maps, we can compute that when $\cot\frac{\epsilon}{c} < |z| < \tan\frac{\epsilon}{c}$, we have $d(f^0(z), y)<\epsilon$ and $d(f^\infty(z), y)<\epsilon$. Thus there exists $\zeta^0(z), \zeta^\infty(z) \in T_y N$ such that $f^0(z)=\exp_y (\zeta^0(z))$ and $f^\infty(z) = \exp_y (\zeta^\infty(z))$.

Consider some fixed nondecreasing smooth function $\rho$ satisfying:
\begin{equation*}
	\rho(z)=\left\{
	\begin{aligned}
		&0,\quad |z|\leq 1,\\
		&1,\quad |z|\geq 2.\\
	\end{aligned}
	\right.
\end{equation*}

We define our pre-glued map $f^R$ by the following formula:
\begin{align}
	\begin{split}
		\label{fR}
		&f^R(z) := f^{\delta,R}(z) =\\
		&\left\{
		\begin{aligned}
			&f^0(z),\quad &|z|\geq \frac{2}{\delta R},\\
			&\exp_y \left(\rho(\delta R z)\zeta^0(z) + \rho\left(\frac{\delta}{Rz}\right)\zeta^\infty(R^2 z)\right),\quad &\frac{\delta}{2R} \leq |z| \leq \frac{2}{\delta R},\\
			&f^\infty(R^2 z),\quad &|z|\leq \frac{\delta}{2R}.
		\end{aligned}
		\right.
	\end{split}
\end{align}

We will need to specify which weighted norm we are using. We use the same weight as on page 376 of McDuff and Salamon \cite[Section 10.3]{mcduff} for the reason mentioned on the same page. Namely, 

\begin{equation*}
	\theta^{R} (z) =\left\{
	\begin{aligned}
		&R^{-2} + R^2 |z|^2,\quad &|z| \leq 1/R,\\
		&1 + |z|^2,\quad &|z| \geq 1/R.\\
	\end{aligned}
	\right.
\end{equation*}

and we let the metric on $S^2$ be
\begin{equation}
	\label{weighted_metric}
	g^R = (\theta^R)^{-2} (ds^2 + dt^2).
\end{equation}
This metric agrees with the metric from $S^2$ by stereographic projection (the Fubini-Study metric) outside radius $1/R$ and, after rescaling, also inside that radius. Thus it defines a metric on $S^2 \#_{\delta, R} S^2$.

For general closed Riemannian manifolds of dimension two, we can choose sufficiently small disks around the point where we intend to glue, and then use the same metrics as above on these disks. This metric will be equivalent to the original Riemannian metrics.

What is different from McDuff and Salamon \cite[Chapter 10]{mcduff} is that, we consider the weighted $W^{2, p}$ norm, where $1 < p < 2$, compared to the weighted $W^{1, p}$ norm where $p > 2$ in McDuff and Salamon. Given the weighted metric in (\ref{weighted_metric}). We define the weighted norms, which we call $(0, p, R)$/$(1, p, R)$/$(2, p, R)$ norms, similar to the definitions on page 376 of McDuff and Salamon \cite{mcduff}. The main difference is that we are only considering sections of the tangent bundles, while McDuff and Salmon also considers 1-forms:
\begin{align}
	&\|\xi\|_{0, p, R} := \left(\int_{\mathbb{C}} \theta^R (z)^{-2} |\xi(z)|^p \right)^{1/p}, \label{0pR}\\
	&\|\xi\|_{1, p, R} := \left(\int_{\mathbb{C}} \theta^R (z)^{-2} |\xi(z)|^p + \theta^R (z)^{p-2} |\nabla \xi(z)|^p \right)^{1/p}, \label{1pR}\\
	&\|\xi\|_{2, p, R} := \left(\int_{\mathbb{C}} \theta^R (z)^{-2} |\xi(z)|^p + \theta^R (z)^{p-2} |\nabla \xi(z)|^p + \theta^R (z)^{2p-2} |\nabla^2 \xi(z)|^p \right)^{1/p}. \label{2pR}
\end{align}
When gluing a Riemann surface $\Sigma_1$ with $S^2$, let $x_1$ be the point where we glue on $\Sigma_1$. We choose a normal coordinate chart $(U_1, \phi_1)$ centered at $x_1$. Since $\Sigma_1$ is compact, we can choose finitely many points $\{p_1, \cdots, p_{n_1}\}$ and corresponding neighborhoods (not containing $x_1$) $\{B_{p_1} (r_1), \cdots, B_{p_{n_1}} (r_1)\}$, such that these open balls together with $U_1$ cover $\Sigma_1$. Denote the metric of $\Sigma_1$ by $g_1$. Let $(s, t)$ be the coordinates of $U_1$, and $z = s + it$, and $x_1 = (0, 0)$ in this coordinate chart. We define the norms to be
\begin{align*}
	\|\xi\|_{0, p, R} &:= \sum_{i = 1}^{n_1} \left(\int_{B_{p_i}(r_1)} |\xi(z)|^p dv_{g_1} \right)^{1/p} + \left(\int_{U_1 \backslash\{|z| \leq 1 / R\}} |\xi(z)|^p dv_{g_1} \right)^{1/p}\\
	& + \left(\int_{\{|z| \leq 1 / R\}} \theta^R(z)^{-2} |\xi(z)|^p dsdt \right)^{1/p},\\
	\|\xi\|_{1, p, R} &:= \sum_{i = 1}^{n_1} \left(\int_{B_{p_i}(r_1)} |\xi(z)|^p + |\nabla \xi(z)|^p dv_{g_1} \right)^{1/p} \\
	&+ \left(\int_{U_1 \backslash\{|z| \leq 1 / R\}} |\xi(z)|^p + |\nabla \xi(z)|^p dv_{g_1} \right)^{1/p} \\
	&+ \left(\int_{\{|z| \leq 1 / R\}} \theta^R (z)^{-2} |\xi(z)|^p + \theta^R (z)^{p-2} |\nabla \xi(z)|^p dsdt\right)^{1/p}, \\
	\|\xi\|_{2, p, R} &:= \sum_{i = 1}^{n_1} \left(\int_{B_{p_i}(r_1)} |\xi(z)|^p + |\nabla \xi(z)|^p + |\nabla^2 \xi(z)|^p dv_{g_1} \right)^{1/p} \\
	&+ \left(\int_{U_1 \backslash\{|z| \leq 1 / R\}} |\xi(z)|^p + |\nabla \xi(z)|^p + |\nabla^2 \xi(z)|^p dv_{g_1} \right)^{1/p} \\
	&\left(\int_{\mathbb{C}} \theta^R (z)^{-2} |\xi(z)|^p + \theta^R (z)^{p-2} |\nabla \xi(z)|^p + \theta^R (z)^{2p-2} |\nabla^2 \xi(z)|^p \right)^{1/p}. 
\end{align*}
We know $g_1$ is uniformly equivalent to
\begin{equation*}
	\theta^R (z) (ds^2 + dt^2)
\end{equation*}
on $U_1 \backslash \{|z| \leq 1 / R\}$. Thus when considering the norms on the neck, we can still use the same metric as if $\Sigma_1$ were $S^2$.

When gluing Riemann surfaces $\Sigma_1$ and $\Sigma_2$, let $x_2$ be the point where we glue on $\Sigma_2$. We choose normal coordinate chart $\{U_2, \phi_2\}$ centered at $x_2$ and open balls $\{B_{q_1} (r_2), \cdots, B_{q_{n_2}} (r_2)\}$ in the same way as for $\Sigma_1$. Let $\tilde{z} = \tilde{s} + i\tilde{t}$ be the coordinates on $U_2$ with $x_2$ as $(0, 0)$. We define the norms to be
\begin{align*}
	\|\xi\|_{0, p, R} &:= \sum_{i = 1}^{n_1} \left(\int_{B_{p_i}(r_1)} |\xi(z)|^p dv_{g_1} \right)^{1/p} + \left(\int_{U_1 \backslash\{|z| \leq 1 / R\}} |\xi(z)|^p dv_{g_1} \right)^{1/p}\\
	& + \sum_{i = 1}^{n_2} \left(\int_{B_{q_i}(r_2)} |\xi(z)|^p dv_{g_1} \right)^{1/p} + \left(\int_{U_2 \backslash\{|\tilde{z}| \leq 1 / R\}} |\xi(z)|^p dv_{g_2} \right)^{1/p},\\
	\|\xi\|_{1, p, R} &:= \sum_{i = 1}^{n_1} \left(\int_{B_{p_i}(r_1)} |\xi(z)|^p + |\nabla \xi(z)|^p dv_{g_1} \right)^{1/p} \\
	&+ \left(\int_{U_1 \backslash\{|z| \leq 1 / R\}} |\xi(z)|^p + |\nabla \xi(z)|^p dv_{g_1} \right)^{1/p} \\
	& + \sum_{i = 1}^{n_2} \left(\int_{B_{q_i}(r_2)} |\xi(z)|^p + |\nabla \xi(z)|^p dv_{g_2} \right)^{1/p} \\
	&+ \left(\int_{U_2 \backslash\{|\tilde{z}| \leq 1 / R\}} |\xi(z)|^p + |\nabla \xi(z)|^p dv_{g_2} \right)^{1/p},\\
	\|\xi\|_{2, p, R} &:= \sum_{i = 1}^{n_1} \left(\int_{B_{p_i}(r_1)} |\xi(z)|^p + |\nabla \xi(z)|^p + |\nabla^2 \xi(z)|^p dv_{g_1} \right)^{1/p} \\
	&+ \left(\int_{U_1 \backslash\{|z| \leq 1 / R\}} |\xi(z)|^p + |\nabla \xi(z)|^p + |\nabla^2 \xi(z)|^p dv_{g_1} \right)^{1/p} \\
	& + \sum_{i = 1}^{n_2} \left(\int_{B_{q_i}(r_2)} |\xi(z)|^p + |\nabla \xi(z)|^p + |\nabla^2 \xi(z)|^p dv_{g_2} \right)^{1/p} \\
	&+ \left(\int_{U_2 \backslash\{|\tilde{z}| \leq 1 / R\}} |\xi(z)|^p + |\nabla \xi(z)|^p + |\nabla^2 \xi(z)|^p dv_{g_2} \right)^{1/p}.\\
\end{align*}
Note that in $\ref{weighted_metric}$, as stated in McDuff and Salamon \cite[Section 10.3, p. 376]{mcduff}, the involution $z \mapsto 1 / (R^2 z)$ is an isometry with respect to the metric $g^R$ that interchanges $\{|z| \leq 1 / R\}$ and $\{|z| \geq 1 / R\}$. Without loss of gerenality, we may assume $U_2$ is an open ball $B_r (0)$ in normal coordinates centered at $x_2$. In this case, the metric $g_2$ on $B_r (0) \backslash \{|z| \leq 1 / R\}$ is uniformly equivalent to the metric $g^R$ on $\{ 1/ (R^2 r) \leq |z| \leq 1 / R\}$ ($R$ can be arbitrarily large) under the above involution. Thus when considering the neck, we can treat $\Sigma_2$ as if it were $S^2$.

In the rest of this section, we will get estimates that will be useful in our construction of the right inverse. Let us denote the $(0, p, R)$-norm (resp. $(2, p, R)$-norm) on a certain region $\Omega$ by $\| \cdot\|_{0, p, R, \Omega}$ (resp. $\|\cdot\|_{2, p, R, \Omega}$).

\subsection{The Harmonic Map Operator}
\label{operator_def}

Let us follow McDuff and Salamon \cite[Chapter 10]{mcduff} and consider gluing harmonic maps on $S^2$ (the Riemann sphere) or any Riemannian manifolds of dimension two into $N$. For simplicity, we consider the case when $\Sigma_1 = S^2, \Sigma_2 = S^2$. However, we will see that the gluing can be used for general Riemann manifolds of dimension two.

First, let us consider the general definition of harmonic maps. Suppose $M$, $N$ are Riemannian manifolds and $f$ is a smooth map from the domain manifold $M$ to the target manifold $N$. We say $f$ is a harmonic map if it satisfies the harmonic map equation. There are multiple ways of writing the harmonic map equation. One would be
\begin{equation}
	P(f):= \triangle_g f + A(f)(df,df),
\end{equation}
where we consider an isometric embedding $N \subset \mathbb{R}^N$ and $A$ is the second fundamental form of the embedding.

However, the above will depend on the ambient Euclidean space. In particular, since we are using the implicit function theorem later, the above will make it difficult to utilize the surjectivity onto the tangent space. An alternate form would be
\begin{equation*}
	\mathrm{tr}_g (\nabla df) = 0.
\end{equation*}
In coordinates, this is
\begin{equation*}
	g^{\alpha\beta} \left( f^k_{\alpha\beta} - (\Gamma^M)^\gamma_{\alpha\beta} f^k_\gamma + (\Gamma^N)^k_{ij} (f) f^i_\alpha f^j_\beta \right) \frac{\p}{\p y^k}= 0,
\end{equation*}
where $x^\alpha, x^\beta$ are coordinates on $S^2$, $y^i, y^j$ are coordinates on $N$, and we are denoting $\frac{\p f}{\p x^\alpha}$ by $f_\alpha$. We know that the above is a well defined section of $f^{-1} TN$, independent of choice of coordinates on $M$ and $N$. We can also verify this by directly computing the transformation law in different coordinates.

We define the operator $P$ as
\begin{equation*}
	P(f) := g^{\alpha\beta} \left( f^k_{\alpha\beta} - (\Gamma^M)^\gamma_{\alpha\beta} f^k_\gamma + (\Gamma^N)^k_{ij} (f) f^i_\alpha f^j_\beta \right) \frac{\p}{\p y^k}.
\end{equation*}
It can be verified directly by coordinate change that the above is a $(0, 1)$ tensor.

Next, we consider the linearization of $P$. The idea is that, for a perturbation of $f$ and $x \in M$, the operator will give a vector that belongs to a different fiber in the tangent bundle of $N$. Thus we have to parallel translate the vectors to the same fiber before comparing them. Similar computation of linearization has been done in the proof of Proposition 3.1 on page 42 of McDuff and Salamon \cite[Section 3.1]{mcduff}.

Here are the details: Given $\xi\in W^{2,p} (\Sigma , f^{-1} TN)$, from page 85 of Adams and Fournier \cite[Theorem 4.12]{adams2003sobolev} we know that for $1 < p < 2$, $W^{2,p}$ is embedded into $C^{0,\gamma}$ where $\gamma = 2 - \frac{2}{p}$. Let
\begin{equation}
	\label{222}
	\Phi_f (\xi): f^{-1} TN\rightarrow (\exp_f(\xi))^{-1} TN
\end{equation}
be the parallel transport. Set
\begin{equation}
	\label{operatorF}
		\mathcal{F}_f(\xi) := \Phi_f(\xi)^{-1}  P (\exp_f (\xi))
\end{equation}
and compute
\begin{equation}
	D_f := d\mathcal{F}_f(0) \label{Df}.
\end{equation}

The following computation of linearization is due to T. Parker. Let $F: M \times (-\epsilon, \epsilon) \times (-\epsilon, \epsilon)$ be a two-parameter variation of a map $f: M\rightarrow N$ (not assumed to be harmonic). Write $F (x, s, t)$ as $f_{s, t} (x)$, so $f = f_{0, 0}$ and set
\begin{equation*}
	X = f_* \frac{\p}{\p s} \quad Y = f_* \frac{\p}{\p t} \quad Z_\alpha = f_* e_\alpha
\end{equation*}
where $\{e_1, e_2\}$ is a local orthonormal frame of $TM$. Then
\begin{equation}
	\label{intrinsic}
	D_f (Y) = \nabla^* \nabla Y + \sum_\alpha R^N (Z_\alpha, Y) Z_\alpha.
\end{equation}
The proof is by T. Parker and will be written out in Appendix \ref{apriori}.

It is well known that a second order elliptic operator from $W^{2, p}$ to $L^p$ spaces of sections of vector bundles on compact manifolds, such as $D_f$, is Fredholm (refer to H\"{o}rmander \cite{hormander1985analysis}). We write it as the following lemma.
\begin{lemma}
	\label{fredholm}
	For any $C^2$ map $f: \Sigma \rightarrow N$, where $\Sigma$ is a Riemann surface and $N$ a closed Riemannian manifold, the operator $D_f$ defined in (\ref{Df}) is a Fredholm operator.
\end{lemma}

\subsection{Structure of the Moduli Space}
\label{moduli}

This part follows from Donaldson and Kronheimer \cite[Section 4.2.4]{donaldson_kronheimer} or McDuff and Salamon \cite[Section A.4]{mcduff}. From Lemma \ref{fredholm} we know that
\begin{equation*}
	D_{f^0}: W^{2, p} \left(S^2, \left(f^0\right)^{-1} TN\right) \rightarrow L^p \left(S^2, \left(f^0\right)^{-1} TN\right)
\end{equation*}
is Fredholm. It follows that the kernel and image of $D_{f^0}$ are closed and admit topological complements. For simplicity, we write
\begin{align*}
	&U := W^{2, p} \left(S^2, \left(f^0\right)^{-1} TN\right),\\
	&V := L^p \left(S^2, \left(f^0\right)^{-1} TN\right).
\end{align*}
So we can write $U = U_0 \oplus F$, $V = V_0 \oplus G$, where $F$ and $G$ are finite-dimensional linear spaces, and $D_{f^0}$ is a linear isomorphism from $U_0$ to $V_0$.

Consider a connected open neighborhood of $0$ in $U$. Since the linearization of $\mathcal{F}_{f^0}$ is Fredholm, we know $\mathcal{F}_{f^0}$ is Fredholm. If $D_{f^0}$ is surjective, then by the implicit function theorem (refer to page 541 of McDuff and Salamon \cite[Theorem A.3.3]{mcduff}), we know there is a diffeomorphism $\psi$ from one neighborhood of $0$ in $U$ to another, such that $\mathcal{F} \circ \psi = D_{f^0}$.

Now consider the general case when $D_{f^0}$ is not necessarily surjective. Consider projection of $V$ onto $V_0$, we will have the derivative be surjective. The following is similar to Theorem A.4.3 on page 546 of McDuff and Salamon \cite[Theorem A.4.3]{mcduff}:
\begin{theorem}
	The Fredholm map $D_{f^0}$ from a neighborhood of $0$ is locally right equivalent to a map of the form
	\begin{equation*}
		\tilde{\mathcal{F}}: U_0 \times F \rightarrow V_0 \times G, \ \tilde{F} (\xi, \eta) = (D_{f^0}(\xi), \alpha(\xi, \eta))
	\end{equation*}
	where $D_{f^0}$ is a linear isomorphism from $U_0$ to $V_0$, $F$ and $G$ are finite-dimensional, and the derivative of $\alpha$ vanishes at $0$.
\end{theorem}
Note that elements in $Z(\mathcal{F})$ are smooth by elliptic regularity (use partition of unity to reduce the case on manifolds to that on Euclidean space).

We know there exists a $C^\infty$ diffeomorphism $g$ from some open set $W$ containing $0 \in W^{2, p} (S^2 , (f^0)^{-1} TN)$, such that
\begin{equation*}
	g(0) = 0,\quad dg(0) = id
\end{equation*}
and
\begin{equation*}
	\mathcal{F}^{-1} (0) \cap g(W) = g(W \cap \ker D)
\end{equation*}

From the above we can get a coordinate chart for $g(W)$ by the isomorphism between $\ker D$ and $\mathbb{R}^m$, where $m$ is the dimension of the kernel. Thus we can view the vectors in the kernel as tangent vectors of the moduli space. Furthermore, a smooth path $f_t$ in the open neighborhood can be represented as $g(\gamma_1 (t) \xi_1 + \cdots + \gamma_m (t) \xi_m)$.

\subsection{The Surjective Operator}
\label{surj_operator}
In McDuff and Salamon \cite[Chapter 10]{mcduff}, the operators are surjective, and one can directly construct a right inverse. There is also work in the context of J-holomorphic curves where the linearization of the Cauchy-Riemann operator is not necessarily surjective. For example, Abouzaid \cite[Lemma 5.2]{abouzaid} considered the right inverse of a restriction of the operator. Hutchings and Taubes \cite[Chapter 2]{taubes_obstruction} considered the obstruction bundle that emerged from the positive dimensional cokernel. McDuff and Wehrheim \cite[Section 4]{mcduff2015kuranishiatlasestrivialisotropy} also considered the obstruction space. Ruan and Tian \cite[Section 6]{ruan_tian} estimated spectrum of certain linear elliptic operators, where the lowest eigenvalue is not bounded away from zero. Nonsurjective cases are also considered in other context. Fukaya, Ono, Oh, and Ohta \cite{fukaya_oh_ohta_ono} considered the obstruction bundle in gluing for Floer theory. In the context of gluing ASD connections, Donaldson and Kronheimer \cite[Proposition 7.2.28, p. 291]{donaldson_kronheimer} enlarged the domain of the operator so that it became surjective. We will use the same idea as in Donaldson and Kronheimer.

Given $(\tilde{f}_1, \tilde{f}_2) \in \mathcal{M} (c, p)$, we will consider a neighborhood $\mathcal{U}$ of this element, which we can make smaller whenever necessary. Now consider the cokernels of $D_{\tilde{f}_1}$ and $D_{\tilde{f}_2}$ as in Definition \ref{moduli_defn_nonsurj}. We will use a method similar to Donaldson and Kronheimer \cite[Proposition 7.2.28, p. 291]{donaldson_kronheimer}.

Recall the representatives of the cokernel defined in Definition \ref{V}.
\begin{definition}
	\label{sigma}
	Let
	\begin{equation*}
		\sigma: \tilde{V} \rightarrow L^p (\Sigma_1, \tilde{f}_1^{-1} TN) \times L^p (\Sigma_2, \tilde{f}_2^{-1} TN),
	\end{equation*}
	be the identity maps. For sufficiently small $\mathcal{U}$ and $(f_1, f_2) \in \mathcal{U}$, let
	\begin{equation*}
		\sigma: \tilde{V} \rightarrow L^p (\Sigma_1, {f}_1^{-1} TN) \times L^p (\Sigma_2, {f}_2^{-1} TN),
	\end{equation*}
	map elements by parallel transport.
\end{definition}
For small enough $\mathcal{U}$, we know that $D_{f_1, f_2} \oplus \sigma$ will be surjective for all $(f_1, f_2) \in \mathcal{U}$.

\begin{definition}
	\label{norm_and_inner_product}
	Set the norm and inner product on $\tilde{V}$ by setting $\{\tilde{v}_{1}, \cdots, \tilde{v}_{k}\}$ as an orthonormal set. Define the norm of $W^{2, p}_{f_1, f_2} \times \tilde{V}$ by
	\begin{equation*}
		\| (\xi, \tilde{v})\|_{W^{2, p}_{f_1, f_2} \times \tilde{V}} := \sqrt{\|\xi\|_{W^{2, p}}^2 + \|\tilde{v}\|_{\tilde{V}}^2}.
	\end{equation*}
	Define the inner product of $W^{2, p}_{f_1, f_2} \times \tilde{V}$ by the $L^2$ inner product in $W^{2, p}_{f_1, f_2}$ and setting $\tilde{V}$ to be orthogonal to $W^{2, p}_{f_1, f_2}$.
\end{definition}
In fact, for finite-dimensional spaces, all norms are equivalent, so this is not essential.

In order to simplify the notations for the $2, p, R$ space direct product with the finite-dimensional spaces $\tilde{V}$, we define the following notation:
\begin{definition}
	\label{weighted_with_V}
	Let
	\begin{equation*}
		\|(\xi, \tilde{v})\|_{2, p, R, V} := \sqrt{\|\xi\|_{2, p, R}^2 + \|\tilde{v}\|_{\tilde{V}}^2}.
	\end{equation*}
\end{definition}

\section{Existence of the Extended Gluing Map}
\label{existence_of_the_gluing_map}

\subsection{The implicit function theorem}
\label{implicit_function_theorem}

We first state a version of implicit function theorem for general Banach spaces:
\begin{proposition}[Implicit Function Theorem for General Banach Spaces](See McDuff and Salamon \cite[Proposition A.3.4, p. 542]{mcduff})
	\label{implicit_general}
	Let $X$ and $Y$ be Banach spaces, $U \subset X$ be an open set, and $f: U \rightarrow Y$ be a continuously differentiable map. Let $x_0 \in U$ be such that $D := df(x_0) : X \rightarrow Y$ is surjective and has a (bounded linear) right inverse $Q: Y\rightarrow X$. Choose positive constants $\delta$ and $c$ such that $\|Q\| \leq c$, $B_\delta (x_0 ; X) \subset U$, and
	\begin{equation}
		\|x - x_0 \| < \delta\ \Rightarrow \| df(x) - D\| \leq \frac{1}{2c}.
	\end{equation}
	Suppose that $x_1 \in X$ satisfies
	\begin{equation}
		\|f(x_1)\| < \frac{\delta}{4c},\ \|x_1 - x_0\| < \frac{\delta}{8}.
	\end{equation}
	Then there exists a unique $x\in X$ such that
	\begin{equation}
		f(x) = 0,\ x - x_1 \in \text{im} Q,\ \|x - x_0\| < \delta.
	\end{equation}
	Moreover, $\|x - x_1\| \leq 2c \|f(x_1)\|$.
\end{proposition}

\subsection{Pregluing}
\label{pregluing}

From now on we consider gluing two harmonic maps from $S^2$ to some closed Riemannian manifold $(N, h)$. However, the same procedure can be applied in the general case. The pregluing is similar to the procedure in McDuff and Salamon \cite[Section 10.5, p. 382]{mcduff}.

More specifically, consider $\mathcal{M} (c, p)$ as in definition \ref{moduli_defn_nonsurj}. Let $f_1, f_2, x_1, x_2$ be\\
$f^0, f^\infty, 0, \infty$ respectively (here we are using the projective plane $\mathbb{CP}^1$ as the coordinate chart for $S^2$). Furthermore, suppose $\epsilon$ is less than the injectivity radius of $N$. Consider the setup in Section \ref{gluing_set_up}.

Let $r := \delta R$. Note that this definition is only for the brevity of notations. In the process of estimates, $r$ will denote the radius for polar coordinates. We will need the following $W^{2,p}$-small perturbations $f^{0,r}, f^{\infty,r}$ of $f^0, f^\infty$:
\begin{equation}
	\label{f0r}
	f^{0,r} (z) := \left\{
	\begin{aligned}
		&f^R (z),\quad &|z| \geq \frac{1}{r},\\
		&f^0 (0),\quad &|z| \leq \frac{1}{r}.
	\end{aligned}
	\right.
	\quad
	f^{\infty,r} (z) := \left\{
	\begin{aligned}
		&f^R \left(\frac{z}{R^2}\right),\quad &|z| \leq r,\\
		&f^\infty (\infty),\quad &|z| \geq r.
	\end{aligned}
	\right.
\end{equation}
Note that $f^{0,r} (z) = f^R (z)$ for all $|z| \geq \delta/R$ and $f^{\infty, r} (z) = f^R (z / R^2)$ for all $|z| \leq R / \delta$.

Let $\mathcal{W}_{f^{0,\infty}} \subset W^{2,p}_{f^{0,\infty}}$ be the $L^2$ orthogonal complement of the kernel of $D_{0,\infty}$, and define
\begin{equation}
	\label{Q0inf}
	Q_{0,\infty} := Q_{f^0,f^\infty} := (D_{0,\infty} \mid_{\mathcal{W}_{f^{0,\infty}}}\oplus \sigma)^{-1}.
\end{equation}
Similarly, we have
\begin{equation}
	\label{Q0infr}
	Q_{0,\infty,r} := Q_{f^{0,r},f^{\infty,r}}.
\end{equation}
Since the operators $D_{0,\infty,r}$ are small perturbations of $D_{0,\infty}$ (see Proposition \ref{D_perturbation}), we know that these right inverses still exist.

We will need the following estimate: There are positive constants $\delta_0, c_0$ only depending on  such that
\begin{equation}
	\label{322}
	\|Q_{0,\infty,r} \eta \|_{W^{2,p}} \leq c_0 \|\eta\|_{L^p},
\end{equation}
for all $(f^0 ,f^\infty) \in \mathcal{M} (c, p)$, and $(\delta, R) \in \mathcal{A} (\delta_0)$, and $(\eta^0, \eta^\infty) \in L^p_{f^{0,r}} \times L^p_{f^{\infty,r}}$. This will be shown in Lemma \ref{Q_est}.

\subsection{Approximate right inverse}
\label{approximate_right_inverse}

The idea of our proof is that, first find an approximate right inverse, then we use this to find the real right inverse and apply the implicit function theorem. In this section, we will define the approximate right inverse and prove certain estimates that will be useful later.

Recall that for any $0 < \delta_0 < 1$, we define $\mathcal{A} (\delta_0)$ to be the set of all pairs of $(\delta, R)$ such that $0 < \delta < \delta_0$ and $\delta R > 1 / \delta_0$. Also recall the definition of $\mathcal{M} (c, p)$ in Definition \ref{moduli_defn_nonsurj}. In the definition below, note that for any pair of positive numbers $\delta_1, \delta_2$ such that $\delta_1 < \delta_2$, we have $\mathcal{A} (\delta_1) \subset \mathcal{A} (\delta_2)$. Thus we can always shrink $\delta_0$ if necessary.

The following construction of the approximate inverse is an analog of the construction on page 382 in McDuff and Salamon \cite[Proposition 10.5.1]{mcduff}.

\begin{definition}
	\label{approx_right_inverse_defn}
	For any $(\tilde{f}^0, \tilde{f}^\infty) \in \mathcal{M} (c, p)$, we can choose a neighborhood $\mathcal{U}$ in $\mathcal{M} (c, p)$ and choose $\delta_0$ to be the same as in \ref{Q_est}, such that, for any $(f^0, f^\infty) \in \mathcal{U}$ and any $(\delta, R) \in \mathcal{A} (\delta_0)$, we define
	\begin{equation*}
		T_{f^R} : L^p (S^2 , (f^R)^{-1} TN) \rightarrow W^{2,p}(S^2 , (f^R)^{-1} TN) \times \tilde{V}
	\end{equation*}
	along the preglued map $f^R : S^2 \rightarrow N$ defined by (\ref{fR}) as follows:
	
	Given $\eta \in L_{f^R}^p$ we first define the pair
	\begin{equation*}
		(\eta^0, \eta^\infty) \in L^p_{f^{0,r}} \times L^p_{f^{\infty,r}}
	\end{equation*}
	by cutting off $\eta$ along the circle $|z| = 1/R$:
	\begin{equation}
		\label{eta_decomposition}
		\eta^0(z) :=\left\{
		\begin{aligned}
			&\eta(z),& \text{ if } |z| \geq 1/R,\\
			&0,& \text{ if } |z| \leq 1/R,
		\end{aligned}
		\right.
		\ 
		\eta^\infty(z) :=\left\{
		\begin{aligned}
			&\eta(z/R^2),& \text{ if } |z| \leq R,\\
			&0,& \text{ if } |z| \geq R.
		\end{aligned}
		\right.
	\end{equation}
	Second, define
	\begin{equation}
		\label{xi}
		(\xi^0 , \xi^\infty, \tilde{v}) := Q_{0,\infty,r}(\eta^0, \eta^\infty)
	\end{equation}
	and note that the vector fields $\xi^0, \xi^\infty$ have the same value $\xi_0$ at the points where two maps meet:
	\begin{equation*}
		\xi^0(0) = \xi^\infty(\infty) =: \xi_0 \in T_{f^0(0)}N.
	\end{equation*}
	Third, let $1 - \beta_{\delta,R} : \mathbb{C} \rightarrow \mathbb{R}$ denote a cutoff function defined as follows: $\beta_{\delta,R}(z) = 0$ for $|z| \leq \delta/R$, and $\beta_{\delta,R} (z) = 1$ for $|z| \geq 1/R$, and
	\begin{equation*}
		\beta_{\delta,R}(z) := \kappa\left(\frac{\log (R|z|/\delta)}{\log (1/\delta)}\right), \quad \frac{\delta}{R} \leq |z| \leq \frac{1}{R}
	\end{equation*}
	where $\kappa$: $\mathbb{R} \rightarrow [0, 1]$ is a $C^\infty$ cut-off function such that $\kappa(t) = 1$ if $t\geq 1$ and $\kappa(t) = 0$ if $t\leq 0$. Fourth, define $T_{f^R} \eta := (\xi^R, \tilde{v})$ by defining
	\begin{equation}
		\label{xiR}
		\xi^R (z) := \left\{
		\begin{aligned}
			&\xi^0 (z),\ & \text{ if } |z| \geq \frac{1}{\delta R},\\
			&\xi^0 (z) + \beta_{\delta,R}\left(\frac{1}{R^2 z}\right) (\xi^\infty (R^2 z) - \xi_0),\ & \text{ if } \frac{1}{R} \leq |z| \leq \frac{1}{\delta R},\\
			&\xi^0 (z) + \xi^\infty (R^2 z) - \xi_0,\ & \text{ if } |z| = \frac{1}{R},\\
			&\xi^\infty (R^2 z) + \beta_{\delta,R}(z) (\xi^0 (z) - \xi_0),\ &\text{ if } \frac{\delta}{R} \leq |z| \leq \frac{1}{R},\\
			& \xi^\infty (R^2 z),\ &\text{ if } |z|\leq \frac{\delta}{R}.
		\end{aligned}
		\right.
	\end{equation}
\end{definition}

The following Lemma \ref{approx_right_inverse_est11} and Lemma \ref{approx_right_inverse_est12} correspond to estimates in the proof of Proposition 10.5.1 on page 382 of McDuff and Salamon \cite[Proposition 10.5.1]{mcduff}.
\begin{lemma}
	\label{approx_right_inverse_est11}
	For any $(\tilde{f}^0, \tilde{f}^\infty) \in \mathcal{M} (c, p)$, we can choose a neighborhood $\mathcal{U}$ of $\mathcal{M} (c, p)$, and $\delta_0$ small enough only depending on $c, p, N, \mathcal{U}$, such that for any $(f^0, f^\infty) \in \mathcal{U}$ and any $(\delta, R) \in \mathcal{A} (\delta_0)$, the approximate right inverse $T_{f^R}$ defined in Definition \ref{approx_right_inverse_defn} satisfies:
	\begin{equation*}
		\|(D_{f^R} \oplus \sigma) T_{f^R} \eta - \eta\|_{0,p,R, \Omega_1} \leq \frac{1}{4}\|\eta\|_{0,p,R, \Omega_1}
	\end{equation*}
	for every $\eta \in L^p (S^2, (f^R)^{-1}TN)$, where $\Omega_1 = \{\delta/ R \leq |z| \leq 1 / R\}$.
\end{lemma}

\begin{proof}
	Recall that the elements in $\tilde{V}$ are supported away from the neck. The proof is elementary. We compute the expressions in coordinates and apply Sobolev embedding and H\"{o}lder's inequality.
	
	In this region,
	\begin{equation*}
		f^{0,r}(z) = f^{\infty,r}(R^2 z) = f^R (z) = y.
	\end{equation*}
	Therefore over this annulus the vector field $\xi^R$ (defined in (\ref{xiR})) takes values in the fixed vector space $T_y N$. Furthermore, the definition of $\xi^\infty$ (defined in (\ref{xi})) implies that $D_{f^R} \xi^\infty (R^2 \cdot) = \eta$ in the region $|z| \leq 1/R$.
	
	Recall the formula for the linearization in Equation (\ref{intrinsic}). We can write the terms in coordinates (refer to Nicolaescu \cite[Example 10.1.32, p. 458]{liviu2007geometryofmanifolds}) as follows.
	\begin{equation}
		\nabla^* \nabla = -\sum_{k, j} \left( g^{kj} \nabla_k \nabla_j + \frac{1}{\sqrt{|g|}} \p_{x^k} (\sqrt{|g|} g^{kj}) \cdot \nabla_j \right)
	\end{equation}
	
	Let's consider the coordinates
	\begin{equation}
		\label{polar}
		z = s + it = r\cos\theta + ir\sin\theta.
	\end{equation}
	We recall that, by our construction of the cutoff function, we have
	\begin{align}
		&\left| \frac{\p}{\p r} \beta (z) \right| \leq \frac{C}{ |z| \log(1 / \delta) },\label{337}\\
		&\left| \frac{\p^2}{\p r^2} \beta (z)\right| \leq \frac{C}{|z|^2 \log(1/\delta)},\label{338}
	\end{align}
	where the constants are universal.
	
	Hence, when $\delta/R \leq |z| \leq 1/R$, note that $D_{f^{0, r}} \xi^0 = \eta^0 = 0$ (see (\ref{eta_decomposition}) and (\ref{xi}) for definitions), we find
	\begin{align}
		D_{f^R} \xi^R - \eta &= D_{f^{0,r}} (\beta_{\delta,R} (\xi^0 - \xi_0))\\
		&= \nabla^* \nabla (\beta_{\delta, R} (\xi^0 - \xi_0)) + \sum_\alpha R^N (Z_\alpha, \beta_{\delta,R} (\xi^0 - \xi_0)) Z_\alpha \label{eq1}\\
		&= - \left(R^{-2} + R^2 |z|^2\right)^2 \left(\nabla_{\frac{\p}{\p s}} \nabla_{\frac{\p}{\p s}} + \nabla_{\frac{\p}{\p t}}\nabla_{\frac{\p}{\p t}}\right) (\beta_{\delta, R} (\xi^0 - \xi_0)) + \sum_\alpha R^N (Z_\alpha, \beta_{\delta,R} (\xi^0 - \xi_0)) Z_\alpha \label{eq2}
	\end{align}
	where $Z_\alpha$ is defined as in Equation (\ref{Zalpha}).
	
	Since $f^R$ is constant for $\delta / R \leq |z| \leq 1 / R$, $Z_\alpha$ in Equation (\ref{eq2}) vanishes for all $\alpha$. We only need to consider the first term in Equation (\ref{eq2}). Without loss of generality, we can only consider estimating
	\begin{align*}
		&- \left(R^{-2} + R^2 |z|^2\right)^2 \nabla_{\frac{\p}{\p s}} \nabla_{\frac{\p}{\p s}}  (\beta_{\delta, R} (\xi^0 - \xi_0))\\
		=& -\left(R^{-2} + R^2 |z|^2\right)^2 \left( \frac{\p \beta(z)}{\p s^2} (\xi^0 - \xi_0) + 2 \frac{\p \beta(z)}{\p s} \frac{\p \xi^0}{\p s} + \beta(z) \frac{\p^2 \xi^0}{\p s^2} \right)
	\end{align*}
	where the equality holds because $f^R$ is constant for $\delta / R \leq |z| \leq 1 / R$.
	
	From Equation (\ref{338}) we know
	\begin{equation*}
		\left|\frac{\p^2 \beta(z)}{\p s^2} \right| \leq \frac{C}{|z|^2 \log(1/\delta)}.
	\end{equation*}
	Furthermore, note that we have control over the $C^{0,\gamma}$ norm of $\xi^0$ by the $W^{2,p}$ norm. We consider the Sobolev embedding of $C^{0,\mu}$ into $W^{2,p}$ where $1 < p < 2$ and $\mu = 2 - \frac{2}{p}$. For every ball $B \subset \mathbb{R}^2$ and every $\xi \in W^{2,p}(B)$, we have
	\begin{align}
		\label{embedding}
		z_0, z_1 \in B \ \Rightarrow\ |\xi(z_1) - \xi(z_0)| \leq C(p) \|\xi\|_{W^{2,p}(B)}|z_1 - z_0|^{2 - \frac{2}{p}}.
	\end{align}
	Note that in the Sobolev embedding, for balls of different radius, the constant remains the same. Thus, under the weighted norm (\ref{0pR}),
	\begin{align*}        
		& \left\|(R^{-2} + R^2 r^2)^2 \frac{\p^2 \beta_{\delta,R}}{\p s^2} (\xi^0 - \xi_0) \right\|_{0,p,R}\\
		\leq & \left\|(R^{-2} + R^2 r^2)^2 \frac{C}{r^2 \log (1/\delta)}(\xi^0 - \xi_0)\right\|_{0,p,R}\\
		= & \left(
		\int_{\frac{\delta}{R} \leq |z| \leq \frac{1}{R}} (R^{-2} + R^2 r^2)^{-2} \left| (R^{-2} + R^2 r^2)^2 \frac{C}{r^2 \log(1/\delta)}(\xi^0 - \xi_0)\right|^p 
		\right)^{\frac{1}{p}}\\
		\leq & \left(
		\int_{\frac{\delta}{R} \leq |z| \leq \frac{1}{R}}\frac{(1 + \delta_0^4)^{2p - 2}  C}{r^{2p}(\log(1/\delta))^p}|\xi^0 - \xi_0|^p 
		\right)^{\frac{1}{p}}\\
		\leq & \frac{(1 + \delta_0^4)^{2 - 2 / p} C^{\frac{1}{p}} }{\log (1/\delta)^{1- \frac{1}{p}}}\|\xi^0\|_{W^{2,p}(B_{{1} / {R}})}.
	\end{align*}
	In the above, the first inequality comes from (\ref{339}) and (\ref{eq1}), and the second inequality comes from (\ref{3312}), and the last inequlity uses (\ref{embedding}).
	\begin{align*}
		&\left\|(R^{-2} + R^2 r^2)^2 \frac{\p \beta_{\delta,R}}{\p s} \frac{\p \xi^0}{\p s}\right\|_{0,p,R} \\
		=& \left( \int_{\delta/R \leq |z| \leq 1/R} (R^{-2} + R^2 |z|^2 )^{-2} \left| (R^{-2} + R^2 r^2)^2 \frac{\p \beta_{\delta,R}}{\p s} \frac{\p \xi^0}{\p s}\right|^p \right)^{1/p}\\
		=& \left( \int_{\delta/R \leq |z| \leq 1/R} (R^{-2} + R^2 |z|^2 )^{2p-2} \left| \frac{\p \beta_{\delta,R}}{\p s} \frac{\p \xi^0}{\p s}\right|^p \right)^{1/p}\\
		\leq & (1 + \delta_0^4)^{2- \frac{2}{p}} \left\| \frac{\p \beta_{\delta,R}}{\p s} \frac{\p \xi^0}{\p s}\right\|_{L^p (\delta / R \leq |z| \leq 1 / R)}
	\end{align*}
	by the fact that
	\begin{equation}
		\label{3312}
		R^{-2} + R^2 |z|^2 \leq \left({\delta\delta_0}\right)^2 + 1 \leq \delta_0^4 + 1
	\end{equation}
	for $\delta / R \leq |z| \leq 1 / R$.
	
	We can consider using the Sobolev embedding on the manifold, which, in our case, is the two-dimensional sphere. First, we use the H\"older's inequality to get:
	\begin{align}
		\label{339}
		\left\|\frac{\p \beta_{\delta,R}}{\p s}\frac{\p \xi^0}{\p s}\right\|_p \leq \left\|\frac{\p \beta_{\delta,R}}{\p s}\right\|_{L^2 (\frac{\delta}{R}\leq |z| \leq \frac{1}{R})}\left\|\frac{\p \xi^0}{\p s}\right\|_{L^q (\frac{\delta}{R}\leq |z|\leq \frac{1}{R})}
	\end{align}
	where $q = \frac{2p}{2 - p} > 2$.
	
	For the $L^q$ norm term in Equation (\ref{339}), we want to use the Sobolev embedding on the sphere. From page 35 of Aubin \cite[Section 2.3]{aubin} we know that for compact manifold the Sobolev embedding holds. Note that here we can treat $\xi^0$ in the same way as real-valued functions on the sphere. $\left| \frac{\p \xi^0}{\p s}\right| \leq C \left| \nabla \xi^0 \right|$, where $C$ is univeral. Thus we have
	\begin{align}
		\left\|\frac{\p \xi^0}{\p s}\right\|_{L^q (\frac{\delta}{R}\leq |z|\leq \frac{1}{R})} \leq C \|\xi^0\|_{W^{2,p}(S^2)} \leq C (c, p, N, \mathcal{U}) \|\eta^0\|_{L^p} \leq C (c, p, N, \mathcal{U}) \|\eta\|_{0,p,R}\label{35}
	\end{align}
	if we apply the forthcoming Lemma \ref{Q_est}. Here we are considering $S^2$ with the round metric of radius one.
	
	We can compute using (\ref{337}) that
	\begin{equation*}
		\left\|\frac{\p \beta_{\delta,R}}{\p s}\right\|_{L^2 (\frac{\delta}{R}\leq |z| \leq \frac{1}{R})} \leq \frac{C}{\sqrt{\log (1/\delta)}}.
	\end{equation*}

	Furthermore, note that,
	\begin{align*}
		\|\xi^0\|_{W^{2,p}(B_{1 / R})} \leq C \|\xi^0\|_{W^{2,p}(S^2)} \leq C (c, p, N, \mathcal{U}) \|\eta^0\|_{L^p} \leq C (c, p, N, \mathcal{U}) \|\eta\|_{0,p,R},
	\end{align*}
	where the second inequality comes from (\ref{322}) and (\ref{xi}), and the third inequality comes from (\ref{eta_decomposition}).
	
	Now we are left with
	\begin{align*}
		&\left\| (R^{-2} + R^2 r^2)^2 \beta_{\delta, R} \frac{\p^2 \xi^0}{\p s^2}\right\|_{0, p, R}\\
		=&  \left( \int_{\delta/R \leq |z| \leq 1/R} (R^{-2} + R^2 |z|^2 )^{-2} \left| (R^{-2} + R^2 r^2)^2 \beta_{\delta,R} \frac{\p^2 \xi^0}{\p s^2}\right|^p \right)^{1/p}\\
		\leq & \left( \int_{\delta/R \leq |z| \leq 1/R} (R^{-2} + R^2 |z|^2 )^{2p-2} \left| \frac{\p \beta_{\delta,R}}{\p s} \frac{\p \xi^0}{\p s}\right|^p \right)^{1/p}\\
		\leq & (1 + \delta_0^4)^{2 - \frac{2}{p}} \left\| \frac{\p^2 \xi^0}{\p s^2}\right\|_{L^p (\delta/ R \leq |z| \leq 1 / R)}.
	\end{align*}
	Similar as in Equation (\ref{35}), we have
	\begin{equation*}
		 \left\| \frac{\p^2 \xi^0}{\p s^2}\right\|_{L^p (\delta/ R \leq |z| \leq 1 / R)} \leq C \|\xi^0\|_{W^{2, p} (S^2)} \leq C (c, p, N, \mathcal{U}) \|\eta^0\|_{L^p} \leq C(c, p, N, \mathcal{U}) \|\eta\|_{0, p, R}.
	\end{equation*}
	Thus by choosing small enough $\delta_0$ only depending on $c, p, N, \mathcal{U}$, we have the desired result.
\end{proof}

The following Lemma \ref{approx_right_inverse_est12} corresponds to an estimate in the proof of Proposition 10.5.1 on page 382 of McDuff and Salamon \cite[Proposition 10.5.1]{mcduff}. This estimate covers the part of the neck that was not in Lemma \ref{approx_right_inverse_est11}.

\begin{lemma}
	\label{approx_right_inverse_est12}
	For any $(\tilde{f}^0, \tilde{f}^\infty) \in \mathcal{M} (c, p)$, we can choose a neighborhood $\mathcal{U}$ of $\mathcal{M} (c, p)$, and $\delta_0$ small enough only depending on $c, p, N, \mathcal{U}$, such that for any $(f^0, f^\infty) \in \mathcal{U}$ and any $(\delta, R) \in \mathcal{A} (\delta_0)$, the approximate right inverse $T_{f^R}$ defined in Definition \ref{approx_right_inverse_defn} satisfies:
	\begin{equation*}
		\|(D_{f^R} \oplus \sigma) T_{f^R} \eta - \eta\|_{0,p,R, \Omega_2} \leq \frac{1}{4}\|\eta\|_{0,p,R, \Omega_2}
	\end{equation*}
	for every $\eta \in L^p (S^2, (f^R)^{-1}TN)$, where $\Omega_2 = \{ 1 / R \leq |z| \leq 1 / \delta R\}$.
\end{lemma}

\begin{proof}
	We still use the coordinates as in Equation (\ref{polar}). Here we have
	\begin{equation*}
		\xi^R (z) = \xi^0 (z) + \beta_{\delta,R}\left(\frac{1}{R^2 z}\right) (\xi^\infty (R^2 z) - \xi_0).
	\end{equation*}
	We know that $D_{f^R}\xi^0 = D_{f^{0,r}} \xi^0 = \eta^0 = \eta$. We also know that $D_{f^R} \xi^\infty(R^2 z) = D_{f^{\infty,r}(R^2 \cdot)}\xi^\infty (R^2 z)$. Again, we use $(s,t)$ coordinates for $1/R \leq |z| \leq 1 / \delta R$ ($z = s + it$). We use $(\tilde{s} = R^2 s,\ \tilde{t} = R^2 t)$ coordinates for $R \leq R^2 |z| \leq R / \delta$. We know that in the annulus $1/ R \leq |z| \leq 1/ \delta R$, we have
	\begin{equation*}
		f^R (z) = f^{\infty,r}(R^2 z) = f^0 (0) = f^\infty (\infty) = y.
	\end{equation*}
	Similar as in the proof of Lemma \ref{approx_right_inverse_est11}, the error is
	\begin{equation*}
		- (1 + |z|^2)^2 \left( \nabla_{\frac{\p}{\p s}} \nabla_{\frac{\p}{\p s}} + \nabla_{\frac{\p}{\p s}} \nabla_{\frac{\p}{\p s}}\right) \left( \beta_{\delta, R} \left(\frac{1}{R^2 z}\right) (\xi^\infty (R^2 z) - \xi_0) \right).
	\end{equation*}
	We only need to estimate the $0, p, R$ norm of
	\begin{align*}
		&(1 + |z|^2)^2 \left( \frac{\p^2}{\p s^2} \left(\beta_{\delta, R} \left(\frac{1}{R^2 z}\right)\right) (\xi^\infty (R^2 z) - \xi_0) + 2 \frac{\p}{\p s}  \left(\beta_{\delta, R} \left(\frac{1}{R^2 z}\right) \right) \frac{\p}{\p s} (\xi^\infty (R^2 z)) +\right.\\ &\left. \beta_{\delta, R} \left(\frac{1}{R^2 z}\right) \frac{\p^2}{\p s^2} (\xi^\infty (R^2 z))\right).
	\end{align*}
	
	We compute that
	\begin{align*}
		\left|\frac{\p}{\p s}\left( \beta_{\delta,R}\left(\frac{1}{R^2 z}\right) \right)\right| \leq & \frac{C}{r\log (1/\delta)},\\
		\left|\frac{\p^2}{\p s^2}\left( \beta_{\delta,R}\left(\frac{1}{R^2 z}\right) \right)\right| \leq &\frac{C}{r^2 \log (1/\delta)},
	\end{align*}
	where $C$ is a universal constant. Let's first estimate
	\begin{align*}
		&\left\|(1 + |z|^2)^2 \frac{\p}{\p s}\left( \beta_{\delta,R}\left(\frac{1}{R^2 z}\right) \right) \frac{\p}{\p s} \left( \xi^\infty(R^2 z) \right)\right\|_{0,p,R}\\
		=& \left( \int_{0}^{2\pi} d\theta \int_{1/R}^{1/(\delta R)} rdr\ (1 + r^2)^{2p - 2} \left| \frac{\p}{\p s}\left(\beta_{\delta,R}\left(\frac{1}{R^2 z}\right)\right)\frac{\p}{\p s}\left( \xi^\infty (R^2 z)\right)\right|^p\right)^{\frac{1}{p}}\\
		\leq & \left( 1 + \left(\frac{1}{\delta R}\right)^2\right)^{2 - \frac{2}{p}}\left( \int_{0}^{2\pi} d\theta \int_{1/R}^{1/(\delta R)} rdr\ \left| \frac{\p}{\p s}\left(\beta_{\delta,R}\left(\frac{1}{R^2 z}\right)\right)\right|^p \left|\frac{\p}{\p s}\left( \xi^\infty (R^2 z)\right)\right|^p\right)^{\frac{1}{p}}\\
		\leq & \left( 1 + \left(\frac{1}{\delta R}\right)^2\right)^{2 - \frac{2}{p}}\left( \int_A \ \left(\left| \frac{\p}{\p s}\left(\beta_{\delta,R}\left(\frac{1}{R^2 z}\right)\right)\right|^p \right)^{\frac{2}{p}} \right)^{\frac{1}{2}} \\
		&\left(\int_A \left(\left|\frac{\p}{\p s}\left( \xi^\infty (R^2 z)\right)\right|^p\right)^{\frac{2}{2 - p}}\right)^{\frac{2 - p}{2p}}\\
	\end{align*}
	where $A$ denotes the annulus $1/R \leq |z| \leq 1 / (\delta R)$. The last inequality comes from H\"{o}lder's inequality.
	
	For the integral involving the cutoff function, we have
	\begin{align*}
		\left( \int_A \ \left(\left| \frac{\p}{\p s}\left(\beta_{\delta,R}\left(\frac{1}{R^2 z}\right)\right)\right|^p \right)^{\frac{2}{p}} \right)^{\frac{1}{2}} \leq & 
		\left( \int_0^{2\pi} d\theta \int_{1/R}^{1/(\delta R)} rdr \ \frac{C}{r^2 (\log (1/\delta))^2} \right)^{\frac{1}{2}}\\
		\leq & C\sqrt{\frac{2\pi}{\log (1/\delta)}}
	\end{align*}
	where $C$ is a universal constant.
	
	Here are some thoughts about measuring the norm of the derivative of $\xi^\infty$. Let us note that with the coordinate chart given by stereographic projection, there is a weight for the derivative of $\xi^\infty$, since in the coordinate chart, at the points away from the origin, say $w \in \mathbb{R}^2\backslash \{0\}$, the vector is dilated with a ratio of $(1 + |w|^2)^{-1}$. We have
	\begin{align*}
		\frac{\p}{\p s}(\xi^\infty (R^2 z)) = R^2 \frac{\p \xi^\infty}{\p s}(R^2 z) \leq (1 + |R^2 z|^2 )\frac{\p \xi^\infty}{\p s}(R^2 z)
	\end{align*}
	
	Recall that using the coordinates from the stereographic projection of the Riemann sphere $S^2$, the $W^{2,p}$ norm is defined as
	\begin{align*}
		\|\xi\|_{W^{2,p}(S^2, f^{-1} TN)} =& \left( \int_{S^2} |\xi|^p + |\nabla\xi|^p + |\nabla^2 \xi |^p \right)^{\frac{1}{p}}\\
		=& \left( \int_{\mathbb{R}^2} (1 + r^2 )^{-2}|\xi|^p + (1 + r^2)^{p - 2} (|\nabla_s \xi|^p + |\nabla_t \xi|^p) +\right.\\
		&\left. (1 + r^2)^{2p - 2} \left( \left| \frac{\p^2 \xi}{\p s^2}\right|^p + \left| \frac{\p^2 \xi}{\p s\p t}\right|^p + \left| \frac{\p^2 \xi}{\p t^2}\right|^p \right)\right)^{\frac{1}{p}} 
	\end{align*}
	Note that here we cannot directly consider the range of $\xi$ to be $\mathbb{R}^N$. By the definition of the $W^{2,p}$ norm for maps mapping to vector bundles, we need to consider the local trivialization of that vector bundle. However, for $\frac{\delta}{R} \leq |z| \leq \frac{1}{\delta R}$, since $f(z) \equiv y$, these two definitions are equivalent.
	
	Denote $\frac{2p}{2 - p}$ by $p^*$. Since $1 < p < 2$, we know $p^* > 2$. Thus it is easy to see that, in the annulus $\frac{1}{R} \leq |z| \leq \frac{1}{\delta R}$,
	\begin{align*}
		R^{2p^*} \leq R^4 ( 1 + R^2 )^{p^* - 2} \leq R^4 (1 + R^4 |z|^2)^{p^* - 2} 
	\end{align*}
	
	We know that
	\begin{align*}
		&\left(\int_A \left(\left|\frac{\p}{\p s}\left( \xi^\infty (R^2 z)\right)\right|^p\right)^{\frac{2}{2 - p}}\right)^{\frac{2 - p}{2p}} \\
		= & \left(\int_A \left| R^2 \frac{\p \xi^\infty}{\p s} (R^2 z)\right|^{p^*}\right)^{\frac{1}{p^*}}\\
		\leq & \left(\int_A R^4 (1 + R^4 |z|^2)^{p^* - 2}\left| \frac{\p \xi^\infty}{\p s} (R^2 z)\right|^{p^*}\right)^{\frac{1}{p^*}}\\
		= & \left(\int_A R^4 (1 + |R^2 z|^2)^{-2} \left| (1 + |R^2 z|^2) \frac{\p \xi^\infty}{\p s} (R^2 z)\right|^{p^*}\right)^{\frac{1}{p^*}}\\
		\leq & \left(\int_{R\leq |w| \leq R/\delta} (1 + |w|^2)^{-2} \left| (1 + |w|^2) \frac{\p \xi^\infty}{\p s} (w)\right|^{p^*}\right)^{\frac{1}{p^*}}\\
		\leq & C \|\nabla \xi^\infty\|_{L^{p^*}(S^2)}
	\end{align*}
	where $C$ is universal.
	
	Then we use the Sobolev embedding theorem for closed Riemann manifolds, as well as Lemma \ref{Q_est}:
	\begin{align*}
		\|\nabla \xi^\infty \|_{L^{p^*}} \leq C \|\xi^\infty\|_{W^{2,p}(S^2)} \leq C(c, p, N, \mathcal{U}) \|\eta^\infty\|_{L^p(S^2)}
	\end{align*}
	
	From the definition of $\eta^\infty$, we have
	\begin{align*}
		\|\eta^\infty\|_{L^p (S^2)} = & \left( \int_{B_R (0)} \left| \eta\left(\frac{z}{R^2}\right)\right|^p (1 + |z|^2)^{-2} \right)^{\frac{1}{p}}\\
		= & \left( R^4 \int_{B_{{1}/{R}} (0)} \left| \eta(w)\right|^p (1 + R^4 |w|^2)^{-2}\right)^{\frac{1}{p}}\\
		= & \left( \int_{B_{{1}/{R}} (0)} \left| \eta(w)\right|^p (R^{-2} + R^2 |w|^2)^{-2}\right)^{\frac{1}{p}}\\
		\leq &\|\eta\|_{0,p,R}
	\end{align*}
	Thus we have
	\begin{align*}
		&\left\|(1 + |z|^2)^2 \frac{\p}{\p s}\left( \beta_{\delta,R} \left(\frac{1}{R^2 z} \right) \right) \frac{\p}{\p s} \left( \xi^\infty(R^2 z) \right)\right\|_{0,p,R}\\
		\leq & C(c, p, N, \mathcal{U}) \left( 1 + \left(\frac{1}{\delta R}\right)^2\right)^{2 - \frac{2}{p}} \sqrt{\frac{2\pi}{\log (1/\delta)}} \|\eta\|_{0,p,R}.
	\end{align*}
	Similarly,
	\begin{align*}
		&\left\|(1 + |z|^2)^2 \beta_{\delta,R} \left(\frac{1}{R^2 z} \right)  \frac{\p^2}{\p s^2} \left( \xi^\infty(R^2 z) \right)\right\|_{0,p,R}\\
		\leq & \left\|(1 + |z|^2)^2  \frac{\p^2}{\p s^2} \left( \xi^\infty(R^2 z) \right)\right\|_{0,p,R}\\
		= & \left(\int_{1 / R \leq |z| \leq 1 / (\delta R)} (1 + |z|^2)^{2p - 2} R^{4p} \left| \frac{\p^2 \xi^\infty}{\p s^2} (R^2 z) \right|^p dsdt \right)^{1/p}\\
		= & \left(\int_{1 / R \leq |z| \leq 1 / (\delta R)} (1 + |z|^2)^{2p - 2} R^{4p} (1 + |R^2 z|^2)^{-2p} \left| (1 + |R^2 z|^2)^2 \frac{\p^2 \xi^\infty}{\p s^2} (R^2 z)\right|^p dsdt \right)^{1 / p}\\
		\leq & \left( 1 + \frac{1}{\delta^2 R^2} \right)^{2 - \frac{2}{p}}  \left(\int_{R \leq |w| \leq R / \delta } R^{4p - 4} ( 1 + |w|^2)^{2 - 2p} \left| (1 + |w|^2)^2 \frac{\p^2 \xi^\infty}{\p s^2} (w) \right|^p \right)^{1 / p}\\
		\leq & (1 + \delta_0^2)^{2 - \frac{2}{p}} \|\xi^\infty\|_{W^{2, p} (S^2)}.
	\end{align*}
	
	On the other hand,
	\begin{align*}
		&\left\| \left( (1 + r^2)^2 \frac{\p^2}{\p s^2} \left(\beta_{\delta,R} \left(\frac{1}{R^2 z}\right) \right) \right) (\xi^\infty (R^2 z) - \xi_0) \right\|_{0,p,R}\\
		\leq & \left\| \frac{C (1 + r^2)^2}{r^2 \log(1/\delta)} (\xi^\infty (R^2 z) - \xi_0) \right\|_{0,p,R}\\
		= & \left( \int_A (1 + r^2)^{-2} \left| \frac{C (1 + r^2)^2}{r^2 \log(1/\delta)} (\xi^\infty (R^2 z) - \xi_0) \right|^p \right)^{\frac{1}{p}}\\
		\leq & \frac{C}{\log (1/\delta)} \left( 1 + \left(\frac{1}{\delta R} \right)^2 \right)^{2 - \frac{2}{p}} \left( \int_A \frac{1}{r^{2p}} | \xi^\infty(R^2 z) - \xi_0 |^p \right)^{\frac{1}{p} }
	\end{align*}
	
	We can consider changing coordinates $w = {1}/{R^2 z}$. Since $\frac{1}{R} \leq |z| \leq \frac{1}{\delta R}$, we have $\frac{\delta}{R} \leq \frac{1}{R^2 z} \leq \frac{1}{R}$. Also, in the new coordinate system, $\xi_0 = \xi^\infty (0)$. We have
	\begin{align}
		\label{embedding_2}
		|\xi^\infty (w) - \xi^\infty (0)| \leq C \| \xi^\infty\|_{W^{2,p}(S^2)} |w|^{2 - \frac{2}{p}}
	\end{align}
	where $C$ is a universal constant.
	
	Thus we have
	\begin{align*}
		&\left( \int_A \frac{1}{r^{2p}} | \xi^\infty(R^2 z) - \xi_0 |^p \right)^{\frac{1}{p} } \\
		\leq & C \|\xi^\infty \|_{W^{2,p}(S^2)} \left( \int_{\frac{1}{R}}^{\frac{1}{\delta R}} \frac{1}{r^{2p}} \left| \frac{1}{R^2 r} \right|^{2p - 2} rdr \right)^{\frac{1}{p}}\\
		= & C \|\xi^\infty \|_{W^{2,p}(S^2)} \left( \int_{\frac{1}{R}}^{\frac{1}{\delta R}} (rR)^{4 - 4p} \frac{1}{r} dr \right)^{\frac{1}{p}}\\
		\leq & C \|\xi^\infty \|_{W^{2,p}(S^2)} \left( \int_{\frac{1}{R}}^{\frac{1}{\delta R}} \frac{1}{r} dr \right)^{\frac{1}{p}}\\
		\leq & C \|\xi^\infty \|_{W^{2,p}(S^2)} \log(1 / \delta)^{\frac{1}{p}}
	\end{align*}
	
	Similar as in the proof of Lemma \ref{approx_right_inverse_est11}, we know
	\begin{align*}
		\|\xi^\infty \|_{W^{2,p}(S^2)} \leq C(c, p, N, \mathcal{U}) \|\eta\|_{0,p,R} 
	\end{align*}
	
	Thus we get
	\begin{align*}
		&\left\| \left( (1 + r^2)^2 \frac{\p^2}{\p s^2} \left(\beta_{\delta,R} \left(\frac{1}{R^2 z}\right) \right) \right) (\xi^\infty (R^2 z) - \xi_0) \right\|_{0,p,R}\\
		\leq & \frac{C(c, p, N, \mathcal{U})}{\log (1/\delta)^{1 - \frac{1}{p}}} \left( 1 + \left(\frac{1}{\delta R} \right)^2 \right)^{2 - \frac{2}{p}}  \|\eta\|_{0,p,R}
	\end{align*}
	
	Hence we can choose $\delta_0$ small enough only depending on $c, p, N, \mathcal{U}$ to get the desired inequality.
\end{proof}

Now we are ready to prove the analog of Proposition 10.5.1 on page 382 of McDuff and Salamon \cite[Proposition 10.5.1]{mcduff}.

\begin{proposition}(Compare McDuff and Salmon \cite[Proposition 10.5.1, p. 382]{mcduff} for the analogous statement for J-holomorphic curves)
	\label{approx_right_inverse}
	For any $(\tilde{f}^0, \tilde{f}^\infty) \in \mathcal{M} (c, p)$, we can choose a neighborhood $\mathcal{U}$ of $\mathcal{M} (c, p)$, and we can choose $\delta_0 > 0$, $c_0 > 0$ only depending on $c, p, N, \mathcal{U}$, such that for any $(f^0, f^\infty) \in \mathcal{U}$ and any $(\delta, R) \in \mathcal{A} (\delta_0)$, the approximate right inverse $T_{f^R}$ defined in Definition \ref{approx_right_inverse_defn} satisfies:
	\begin{equation}
		\label{T_est}
		\|(D_{f^R} \oplus \sigma) T_{f^R} \eta - \eta\|_{0,p,R} \leq \frac{1}{2}\|\eta\|_{0,p,R},\quad \|T_{f^R} \eta\|_{2,p,R, V} \leq \frac{c_0}{2}\|\eta\|_{0,p,R}
	\end{equation}
	for every $\eta \in L^p (S^2, (f^R)^{-1}TN)$.
\end{proposition}
Recall that the $2, p, R, V$ norm is defined in Definition \ref{weighted_with_V}.

\begin{proof}
	We have $T_{f^R} \eta = (\xi^R, \tilde{v})$ for each $\eta \in L^p_{f^R}$, and must prove that
	\begin{equation}
		\| D_{f^R} \xi^R + v - \eta \|_{0,p,R} \leq \frac{1}{2}\|\eta\|_{0,p,R},
	\end{equation}
	where $v = \sigma (\tilde{v})$.
	
	Since $D_{f^{0,r}} \xi^0 + v = \eta^0$ and $D_{f^{\infty,r}}\xi^\infty + v = \eta^\infty$, the term on the left hand side vanishes for $|z| \geq 1/\delta R$ and for $|z| \leq \delta/R$. For $\delta / R \leq |z| \leq 1 / \delta R$, we can apply Lemma \ref{approx_right_inverse_est11} and Lemma \ref{approx_right_inverse_est12}. The first equality is proved.
	
	Before we prove the second inequality, we take a closer look of how the $W^{2, p}$ norm is defined on the weighted sphere.
	
	We still consider the stereographic projection of the weighted $S^2$ (which is the connected sum of the original two $S^2$s). We consider $\xi \in W^{2,p}(S^2, (f^R)^{-1}TN)$.
	
	For $|z| \geq \frac{1}{R}$, we have
	\begin{equation*}
		f^R(z) = \left\{
		\begin{aligned}
			&f^0 (z),\ &&|z| \geq \frac{2}{\delta R}\\
			&\exp_y (\rho(\delta R z) \zeta^0 (z)),\ &&\frac{1}{R} \leq |z| \leq \frac{2}{\delta R}
		\end{aligned}
		\right.
	\end{equation*}
	Since $\mathcal{U}$ is precompact, we know that the norm of  $d\exp_y$ and $d\exp_y^{-1}$ are uniformly bounded. Thus, there exists $c_1, c_2$ only depending on $c, p, N, \mathcal{U}$ such that
	\begin{equation*}
		c_1 \leq \|d\exp_y \| \leq c_2, \quad c_1 \leq \|d\exp_y^{-1} \| \leq c_2
	\end{equation*}
	Furthermore, since $f^0 (z) = \exp_y \zeta^0 (z)$, we have $\zeta^0 (0) = 0$, $|\zeta^0 (z) | \leq C \sup \|df^0\| |z| \leq C(c) |z|$, and $|\nabla \zeta^0 (z)| \leq C(c, p, N, \mathcal{U}) \sup \|df^0\|$.  Thus for $\frac{1}{R} \leq |z| \leq \frac{2}{\delta R}$ we have
	\begin{align*}
		\sup \|d f^R \| \leq C(c, p, N, \mathcal{U})
	\end{align*}
	
	We know that the image of $|z| \geq \frac{1}{R}$ under $f^R$ is the same as the image of $S^2$ under $f^{0,r}$, which is the same as the image of $S^2$ under $f^0$. Since $S^2$ is compact and $\mathcal{U}$ is precompact, we know there is a uniform injective radius only depending on $c, p, N, \mathcal{U}$. 
	
	We can choose coordinate charts as follows: We pick a set of points on $S^2$ and a radius $\sigma < \pi / 2$ such that the geodesic balls of radius $\sigma$ and centered at those points yield a covering of $S^2$. Furthermore, since $\|df^R\|$ is bounded, we can choose $\sigma$ small enough so that the following holds: For any of these geodesic balls, say $B_\sigma (x_1)$, let the radius of the image be less than the injective radius of $f^0 (S^2)$. Now consider the normal coordinates at $f^R(x_1)$, we have a coordinate chart for the image of $B_\sigma (x_1)$. We know that this coordinate chart will only depend on $c, p, N, \mathcal{U}$.
	
	Now on each of these coordinate charts, we can consider the vectors in coordinates, and thus talk about the derivatives in coordinates. That is how we define the $W^{2,p}$ space for $|z| \geq \frac{1}{R}$. Since we choose $\sigma < \pi / 2$, we know that in all these coordinate charts, if we consider the Riemann metric matrix $g_N$ of $N$, we have $0 < C^{-1} \leq \|g_N \| \leq C$ where $C$ is a universal constant. Thus we know that this norm is equivalent to the $2,p,R$ norm, which is formed by using the stereographic projection coordinate chart with the weight.
	
	Next, we consider $|z| \leq \frac{1}{R}$:
	\begin{equation*}
		f^R(z) = \left\{
		\begin{aligned}
			&f^\infty \left(R^2 z\right),\ &&|z| \leq \frac{\delta}{2R}\\
			&\exp_y \left(\rho\left(\frac{\delta}{Rz}\right) \zeta^\infty (R^2 z)\right),\ &&\frac{\delta}{2R} \leq |z| \leq \frac{1}{R}
		\end{aligned}
		\right.
	\end{equation*}
	
	By definition,
	
	\begin{equation*}
		f^{\infty,r}(z) = \left\{
		\begin{aligned}
			&f^\infty (\infty),\ &&|z| \geq \delta R\\
			&\exp_y \left(\rho\left(\frac{\delta R}{z}\right) \zeta^\infty (z)\right),\ &&\frac{\delta R}{2} \leq |z| \leq \delta R\\
			&f^{\infty} (z),\ &&|z| \leq \frac{\delta R}{2}
		\end{aligned}
		\right.
	\end{equation*}
	
	For $f^{\infty,r} (z)$, we can consider the coordinate change $w = \frac{1}{z}$ and do the same as above for $f^{0,r}$. We want to show that this norm is equivalent to the $2,p,R$ norm for $|z| \leq \frac{1}{R}$. Namely, for $\xi^\infty \in W^{2,p}$ defined on $|z| \leq R$ (in $S^2$, not $\mathbb{R}^2$), we want to show that $\|\xi^\infty (R^2 z)\|_{2,p,R}$ is equivalent to $\|\xi^\infty\|_{W^{2,p}}$.
	
	From the construction of the $W^{2,p}$ norm (details written in the case $|z| \geq \frac{1}{R}$), we know the $W^{2,p}$ norm is equivalent to
	\begin{align*}
		\|\xi\|_{W^{2,p}(|z| \leq R)} =& \left( \int_{|z| \leq R} (1 + r^2 )^{-2}|\xi|^p + (1 + r^2)^{p - 2} (|\nabla_s \xi|^p + |\nabla_t \xi|^p) +\right.\\
		&\left. (1 + r^2)^{2p - 2} \left( \left| \frac{\p^2 \xi}{\p s^2}\right|^p + \left| \frac{\p^2 \xi}{\p s\p t}\right|^p + \left| \frac{\p^2 \xi}{\p t^2}\right|^p \right)\right)^{\frac{1}{p}} 
	\end{align*}
	We can use change of variables in the integration to directly verify that this is equivalent to the $2,p,R$ norm. For example, consider $\nabla \xi$ and the change of variables $w = R^2 z$ we have
	\begin{align*}
		&\int_{|z| \leq \frac{1}{R}} (R^{-2} + |z|^2 R^2 )^{-2 + p} |R^2 \nabla \xi^\infty(R^2 z)|^p \\
		=& R^{-4} \int_{|w| \leq R} (R^{-2} + |w|^2 R^{-2})^{-2 + p} R^{2p} |\nabla \xi^\infty (w)|^p\\
		=& \int_{|w| \leq R} (1 + |w|^2)^{-2 + p} |\nabla \xi^\infty(w)|^p
	\end{align*}
	
	Now we have shown that $W^{2,p}(S^2)$ norm is equivalent to $2,p,R$ norm, where $S^2$ is the weighted sphere. In particular, $L^p (S^2)$ norm is equivalent to $0,p,R$ norm.
	
	Now let's come back to showing
	\begin{align*}
		\|T_{f^R} \eta\|_{2,p,R, V} \leq \frac{c_0}{2}\|\eta\|_{0,p,R}
	\end{align*}
	
	From Equation (\ref{xiR}) and what we proved above, we know that we are only left to consider $\beta_{\delta,R}(\frac{1}{R^2 z})(\xi^\infty(R^2 z) - \xi_0)$ for $\frac{1}{R} \leq |z| \leq \frac{1}{\delta R}$ and $\beta_{\delta, R}(z) (\xi^0(z) - \xi_0)$ for $\frac{\delta}{R}\leq |z| \leq \frac{1}{R}$. 
	
	For $\frac{\delta}{R} \leq |z| \leq \frac{1}{R}$, when we consider the $2,p,R$ norm, for parts where there is no derivative on $\beta$, the second order derivative term can be controlled directly. The other parts can be estiamted in the same way as when there is a derivative of $\beta$. Thus we only need to consider the terms (note that we will have to power by $\frac{1}{p}$, which is not written out in the formula)
	\begin{align*}
		\int_{\frac{\delta}{R} \leq |z| \leq \frac{1}{R}} & (R^{-2} + R^2 |z|^2 )^{p - 2}|\nabla \beta_{\delta,R}(z) (\xi^0 (z) - \xi_0)|^p \\
		&+ (R^{-2} + R^2 |z|^2 )^{2p - 2} (|\nabla^2 \beta_{\delta,R}(z) (\xi^0 (z) - \xi_0)|^p + |\nabla \beta_{\delta,R}(z) \nabla\xi^0 (z)|^p)
	\end{align*}
	For the first two parts of the above formula, we use estimates for derivatives of $\beta$ as well as Equation (\ref{embedding}). For the last part, the estimate is similar to the estimate immediately above Equation (\ref{embedding}). For $\frac{1}{R} \leq |z| \leq \frac{1}{\delta R}$, for parts where there is no derivative of $\beta$, we can just estimate
	\begin{align*}
		\int_{\frac{1}{R} \leq |z| \leq \frac{1}{\delta R}} & (1 + |z|^2)^{-2} |\xi^\infty (R^2 z) - \xi_0 |^p + (1 + |z|^2)^{p-2} R^{2 p} |\nabla\xi^\infty (R^2 z)|^p +\\ & (1 + |z|^2)^{2p-2} R^{4p} |\nabla^2 \xi^\infty (R^2 z)|^p
	\end{align*}
	Similar as before, the second derivative term can be controlled directly. The other parts can be controlled in the same way as when there is a derivative of $\beta$.
	
	Now we are only left with
	\begin{align*}
		\int_{\frac{1}{R} \leq |z| \leq \frac{1}{\delta R}} & (1 + |z|^2 )^{p - 2} \left|\nabla \beta_{\delta,R}\left(\frac{1}{R^2 z}\right) \frac{1}{R^2 |z|^2} (\xi^\infty (R^2 z) - \xi_0)\right|^p \\
		& + (1 + |z|^2 )^{2p - 2} \left|\nabla \beta_{\delta,R} \left(\frac{1}{R^2 z}\right) \frac{2}{R^2 |z|^3} (\xi^\infty (R^2 z) - \xi_0) \right|^p\\
		&+ (1 + |z|^2 )^{2p - 2} \left(\left|\nabla^2 \beta_{\delta,R} \left(\frac{1}{R^2 z}\right) \frac{1}{R^4 |z|^4} (\xi^\infty (z) - \xi_0)\right|^p\right.\\
		&+ \left.\left|\nabla \beta_{\delta,R} \left(\frac{1}{R^2 z}\right) \frac{1}{|z|^2} \nabla\xi^\infty (R^2 z)\right|^p\right)
	\end{align*}
	The estimate for the first three parts is the same: we use Equation (\ref{embedding_2}) and change of coordinates. The estimate for the last term is the same as estimating
	\begin{align*}
		&\left\|(1 + |z|^2)^2 \frac{\p}{\p r}\left( \beta_{\delta,R} \left(\frac{1}{R^2 z} \right) \right) \frac{\p}{\p r} \left( \xi^\infty(R^2 z) \right)\right\|_{0,p,R}\\
		\leq & C (c, p, N, \mathcal{U}) \left( 1 + \left(\frac{1}{\delta R}\right)^2\right)^{2 - \frac{2}{p}} \sqrt{\frac{2\pi}{\log (1/\delta)}} \|\eta\|_{0,p,R}
	\end{align*}
	
	The proof is complete.
\end{proof}

From the above Proposition \ref{approx_right_inverse} we can construct the true right inverse. This is the same idea as on page 387 of McDuff and Salamon \cite[Section 10.5]{mcduff}
\begin{definition}
	\label{real_right_inverse}
	For any $(\tilde{f}^0, \tilde{f}^\infty) \in \mathcal{M} (c, p)$, we can choose a neighborhood $\mathcal{U}$ of $\mathcal{M} (c, p)$, and we can choose $\delta_0 > 0$ only depending on $c, p, N, \mathcal{U}$, such that for any $(f^0, f^\infty) \in \mathcal{U}$ and any $(\delta, R) \in \mathcal{A} (\delta_0)$, we define
	\begin{equation*}
		Q_{f^R} := T_{f^R} ((D_{f^R} \oplus \sigma) T_{f^R})^{-1} = \sum_{k = 0}^\infty T_{f^R} (\mathbbm{1} - (D_{f^R} \oplus \sigma) T_{f^R})^k.
	\end{equation*}
\end{definition}

Now we have
\begin{align*}
	(D_{f^R} \oplus \sigma) Q_{f^R} =& \mathbbm{1},\\
	\| Q_{f^R} \eta\|_{2,p,R, V} \leq & c_0 \|\eta\|_{0,p,R},
\end{align*}
where $c_0$ only depends on $c, p, N, \mathcal{U}$.

\subsection{Construction of the gluing map}
\label{construction_of_the_gluing_map}

Let us further tailor Proposition \ref{implicit_general} for our setting:
\begin{theorem}
	\label{tailored_IFT}
	For any $(\tilde{f}^0, \tilde{f}^\infty) \in \mathcal{M} (c, p)$, we can choose a neighborhood $\mathcal{U}$ of $\mathcal{M} (c, p)$, and $\delta_0$ as in Definition \ref{approx_right_inverse_defn}. Consider $(\delta, R) \in \mathcal{A} (\delta_0)$. Let $\Sigma$ denote $S^2$ with the weighted metric defined in the pregluing in Equation (\ref{weighted_metric}). Let $f = (f^0, f^\infty) \in \mathcal{U}$, consider Banach spaces $X = W^{2,p} (\Sigma, \left(f^R\right)^{-1} TN) \times \tilde{V}$ and $Y = L^p (\Sigma, \left(f^R\right)^{-1} TN)$. Let $U$ be an open subset of $X$. Suppose we have the following:
	
	\begin{enumerate}
		\item Consider $0 \in U$, $D_{f^R} \oplus \sigma := d\mathcal{F}_{f^R} (0) \oplus \sigma$ is surjective and has a linear right inverse $Q$ such that $\|Q\| \leq \tilde{c}$ for some constant $\tilde{c}$.
		\item There exists a positive constant $\epsilon$ such that $B_\epsilon (0,X) \subset U$, and
		\begin{equation*}
			\|d\mathcal{F}_{f^R} (\xi) - D_{f^R}\| \leq \frac{1}{2\tilde{c}}
		\end{equation*}
		for all $\|\xi\| < \epsilon$.
		\item There exists some $(\xi_1, \tilde{v}_1) \in X$ that satisfies
		\begin{equation*}
			\|\mathcal{F}_{f^R} (\xi_1) + v_1\| < \frac{\epsilon}{4\tilde{c}},\quad \|(\xi_1, \tilde{v}_{1})\| < \frac{\epsilon}{8},
		\end{equation*}
		where ${v}_{1} = \sigma (\tilde{v}_{1})$.
	\end{enumerate}
	Then there exists a unique $(\xi, \tilde{v}) \in X$ such that
	\begin{align*}
		&\mathcal{F}_{f^R}(\xi) + v = 0,\\
		&(\xi - \xi_1, \tilde{v} - \tilde{v}_{1}) \in \mathrm{im} Q,\\
		&\|(\xi, \tilde{v})\| < \epsilon,
	\end{align*}
	where $v = \sigma (\tilde{v})$.
	
	Moreover, $\|(\xi - \xi_1, \tilde{v} - \tilde{v}_{1})\| \leq 2\tilde{c}\|\mathcal{F}_{f^R}(\xi_1) + v_1 \|$.
\end{theorem}
Recall the norms are defined in Definition \ref{norm_and_inner_product} and Definition \ref{weighted_with_V}. Similar as the argument in McDuff and Salamon \cite[Section 10.5, p. 387]{mcduff}, we now apply Theorem \ref{tailored_IFT} to obtain the gluing map.

We can choose $(\xi_1, \tilde{v}_1)$ to be $0$, and in this chapter we will show that the conditions in the above theorem are satisfied.

First, let us estimate the norm of $\mathcal{F}_{f^R} (0)$ so that we know what $\epsilon/4\tilde{c}$ should be in the third condition of the theorem.

From (\ref{fR}) we know that $\mathcal{F}_{f^R}(0) = 0$ for $|z| \geq \frac{2}{\delta R}$ and $|z| \leq \frac{\delta}{2R}$.

Recall that for normal coordinates, we have the Taylor expansion of the metric:

\begin{equation*}
	g_{ij}(y^1,\cdots,y^n) = \delta_{ij} - \frac{1}{3} y^k y^l R_{iklj} + O (|y|^3)
\end{equation*}

and $\mathcal{F}$ can be written as:
\begin{equation*}
	\mathcal{F}_{f^R} (0)^k = g^{\alpha\beta} \frac{\p^2 (f^R)^k}{\p x^\alpha \p x^\beta} - g^{\alpha\beta} (\Gamma^M)^\gamma_{\alpha\beta} \frac{\p (f^R)^k}{\p x^\gamma} + g^{\alpha\beta} (\Gamma^N)^k_{ij} (f^R) \frac{\p (f^R)^i}{\p x^\alpha} \frac{\p (f^R)^j}{\p x^\beta}
\end{equation*}
for $k = 1, \cdots, n$.

For $\frac{\delta}{R} \leq |z| \leq \frac{1}{\delta R}$, we know that $f^R$ is constant, so we have
\begin{equation*}
	\mathcal{F}_{f^R} (0) = P(f^R) = 0.
\end{equation*}
For $\frac{1}{\delta R} \leq |z| \leq \frac{2}{\delta R}$, we have:
\begin{align*}
	f^R(z) =& \exp_y (\rho (\delta R z) \zeta^0 (z))\\
	\frac{\p f^R}{\p s} =& d\exp_y (\rho(\delta R z)\zeta^0 (z) ) \cdot \left[ \delta R \frac{\p \rho}{\p s}(\delta R z) \zeta^0 (z) + \rho(\delta Rz)\frac{\p \zeta^0(z)}{\p s}\right]\\
	\frac{\p^2 f^R}{\p s^2} =& d\exp_y (\rho(\delta R z) \zeta^0(z)) \cdot \left[ \delta^2 R^2 \frac{\p^2 \rho}{\p s^2} (\delta R z) \zeta^0(z) +\right.\\
	&\left. 2 \delta R\frac{\p \rho}{\p s}(\delta R z) \frac{\p \zeta^0(z)}{\p s} + \rho(\delta R z) \frac{\p^2 \zeta^0(z)}{\p s^2}\right] +\\
	& d^2 \exp_y (\rho(\delta R z)\zeta^0 (z)) \left[ \delta R \frac{\p \rho}{\p s}(\delta R z) \zeta^0 (z) + \rho(\delta Rz)\frac{\p \zeta^0(z)}{\p s}\right]^2\\
	\mathcal{F}_{f^R} (0)^k =& (1 + |z|^2)^2 \left(\frac{\p^2 (f^R)^k}{\p s^2} + \frac{\p^2 (f^R)^k}{\p t^2}\right) + \\
	& (1 + |z|^2)^2(\Gamma^N)_{ij}^k (f^R) \left(\frac{\p (f^R)^i}{\p s} \frac{\p (f^R)^j}{\p s} + \frac{\p (f^R)^i}{\p t} \frac{\p (f^R)^j}{\p t}\right)\\
	\|\mathcal{F}_{f^R}(0)\| =& \|\mathcal{F}_{f^R}(0)\|_{0,p,R} = \left(\int_{\frac{1}{\delta R} \leq |z| \leq \frac{2}{\delta R}} (1 + |z|^2)^{-2} |\mathcal{F}_{f^R} (0)|^p dsdt\right)^{\frac{1}{p}}
\end{align*}

We have $\delta R > \frac{1}{\delta_0}$. For $\delta_0 = \delta_0 (c, p, N, \mathcal{U})$ small enough, the exponential map on $N$ is a smooth isometry for $|z| \leq \frac{2}{\delta R}$. We have
\begin{align*}
	& 0 < c_1 (c, p, N, \mathcal{U}, \delta_0) \leq \| d\exp_y ( \zeta^0 (z))\| \leq c_2 (c, p, N, \mathcal{U}, \delta_0),\  \forall |z| \leq \frac{2}{\delta R}\\
	& \| d^2 \exp_y ( \zeta^0 (z))\| \leq c_3 (c, p, N, \mathcal{U}, \delta_0),\ \forall |z| \leq \frac{2}{\delta R}\\
	& 0 < c_1 (c, p, N, \mathcal{U}, \delta_0) \leq \| d\exp_y (\rho(\delta Rz) \zeta^0 (z))\| \leq c_2 (c, p, N, \mathcal{U}, \delta_0),\  \forall |z| \leq \frac{2}{\delta R}\\
	& \| d^2 \exp_y (\rho(\delta Rz) \zeta^0 (z))\| \leq c_3 (c, p, N, \mathcal{U}, \delta_0),\ \forall |z| \leq \frac{2}{\delta R}
\end{align*}
Note that we may further assume $c_1$ is decreasing with respect to $\delta_0$ while $c_2$ and $c_3$ are increasing.

Since $f^0 (z) = \exp_y (\zeta^0(z))$, and the exponential map is a smooth isometry for $|z| \leq \frac{2}{\delta R}$, we have $\zeta^0(z) = \exp_y^{-1}(f^0(z))$, and
\begin{align*}
	\| d \exp_y^{-1} (f^0 (z))\| \leq & 1/c_1 (c, p, N, \mathcal{U}, \delta_0),\ \forall |z| \leq \frac{2}{\delta R}\\
	\| d^2 \exp_y^{-1} (f^0 (z))\| \leq & c_4 (c, p, N, \mathcal{U}, \delta_0),\ \forall |z| \leq \frac{2}{\delta R}\\
	\left| \zeta^0 (z) \right| \leq & \frac{c}{c_1 (c, p, N, \mathcal{U}, \delta_0)} |z|,\ \forall |z| \leq \frac{2}{\delta R}\\
	\frac{\p \zeta^0 (z)}{\p s} = & \frac{\p \exp_y^{-1}(f^0(z))}{\p s} = d \exp_y^{-1}(f^0(z))\frac{\p f^0 (z)}{\p s}\\
	\frac{\p^2 \zeta^0 (z)}{\p s^2} = & \frac{\p}{\p s} \left( d \exp_y^{-1} (f^0(z))\frac{\p f^0(z)}{\p s}\right) \\
	=& d^2 \exp^{-1}_y (f^0(z)) \left( \frac{\p f^0(z)}{\p s}\right)^2 + d \exp_y^{-1} (f^0 (z)) \frac{\p^2 f^0 (z)}{\p s^2}
\end{align*}
Similar to the above paragraph, we may assume $c_4$ is increasing with respect to $\delta_0$.

We now start to divide $\mathcal{F}_{f^R}(0)$ into terms and estimate each term separately:

(Note that here $\rho$ is a function defined on $\mathbb{R}^2$ instead of $\mathbb{R}$, which is different from the $\rho$ we started with. Actually it is $\rho$ (that we started with) composed with absolute value function)

\begin{align*}
	\left|\frac{\p^2 f^R}{\p s^2} \right|=& \left|d\exp_y (\rho(\delta R z) \zeta^0(z)) \cdot \left[ \delta^2 R^2 \frac{\p^2 \rho}{\p s^2} (\delta R z) \zeta^0(z) +\right.\right.\\
	&\left.2 \delta R\frac{\p \rho}{\p s}(\delta R z) \frac{\p \zeta^0(z)}{\p s} + \rho(\delta R z) \frac{\p^2 \zeta^0(z)}{\p s^2}\right] +\\
	& \left. d^2 \exp_y (\rho(\delta R z)\zeta^0 (z)) \left[ \delta R \frac{\p \rho}{\p s}(\delta R z) \zeta^0 (z) + \rho(\delta Rz)\frac{\p \zeta^0(z)}{\p s}\right]^2\right|\\
	\leq & C (c, p, N, \mathcal{U}, \delta_0) \left[ \delta R + 1 \right]
\end{align*}
where $C$ is increasing with respect to $\delta_0$.

Things are the same for derivatives with respect to $t$ instead of $s$.

\begin{align*}
	\left| \frac{\p (f^R)}{\p s}\right| = & d\exp_y (\rho(\delta R z)\zeta^0 (z) ) \cdot \left[ \delta R \frac{\p \rho}{\p s}(\delta R z) \zeta^0 (z) + \rho(\delta Rz)\frac{\p \zeta^0(z)}{\p s}\right]\\
	\leq & C (c, p, N, \mathcal{U}, \delta_0)
\end{align*}
where $C$ is increasing with respect to $\delta_0$.

We can estimate from the formula of $\mathcal{F}_{f^R}(0)$ that 
\begin{align*}
	|\mathcal{F}_{f^R} (0)| \leq & C (c, p, N, \mathcal{U}, \delta_0) (\delta R + 1)\\
	\|\mathcal{F}_{f^R} (0)\|_{0,p,R} \bigg\rvert_{\{\frac{1}{\delta R} \leq |z| \leq \frac{2}{\delta R}\}} = & \left(\int_{\frac{1}{\delta R} \leq |z|\leq \frac{2}{\delta R}} (1 + |z|^2)^{-2} |\mathcal{F}_{f^R} (0)|^p \right)^{\frac{1}{p}}\\
	\leq & C (c, p, N, \mathcal{U}, \delta_0) \frac{(\delta R +1)}{(\delta R)^{\frac{2}{p}}}
\end{align*}
where $C (c, p, N, \mathcal{U}, \delta_0)$ here is increasing with respect to $\delta_0$. 

Similar as we did for $\frac{1}{\delta R} \leq |z| \leq \frac{2}{\delta R}$, for $\delta_0 = \delta_0 (c, p, N, \mathcal{U})$ small enough, we have:
\begin{align*}
	& 0 < c_1 (c, p, N, \mathcal{U}, \delta_0) \leq \| d\exp_y (\zeta^\infty (R^2 z))\| \leq c_2 (c, p, N, \mathcal{U}, \delta_0),\ \forall \frac{\delta}{2R} \leq |z| \leq \frac{\delta}{R},\\
	&\| d^2 \exp_y(\zeta^\infty (R^2 z))\| \leq c_3 (c, p, N, \mathcal{U}, \delta_0),\ \forall \frac{\delta}{2R} \leq |z| \leq \frac{\delta}{R},\\
	& 0 < c_1 (c, p, N, \mathcal{U}, \delta_0) \leq \left\| d\exp_y \left(\rho \left(\frac {\delta}{Rz}\right)\zeta^\infty (R^2 z)\right) \right\| \leq c_2 (c, p, N, \mathcal{U}, \delta_0),\ \forall \frac{\delta}{2R} \leq |z| \leq \frac{\delta}{R},\\
	&\| d^2 \exp_y \left( \rho \left( \frac{\delta}{Rz} \right)\zeta^\infty (R^2 z) \right)\| \leq c_3 (c, p, N, \mathcal{U}, \delta_0),\ \forall \frac{\delta}{2R} \leq |z| \leq \frac{\delta}{R},
\end{align*}
where $c_1$ is decreasing with respect to $\delta_0$ and $c_2, c_3$ are increasing.

Since $f^\infty (R^2 z) = \exp_y (\zeta^\infty (R^2 z))$ for $\frac{\delta}{2R} \leq |z| \leq \frac{\delta}{R}$, similar to the case when $\frac{1}{\delta R} \leq |z| \leq \frac{2}{\delta R}$, we have:
\begin{align*}
	\zeta^\infty (R^2 z) =& \exp_y^{-1} (f^\infty(R^2 z))\\
	\| d\exp_y^{-1} (f^\infty(R^2 z) \| \leq & 1/c_1 (c, p, N, \mathcal{U}, \delta_0),\ \forall |z| \geq \frac{\delta}{2R},\\
	\| d^2 exp_y^{-1} (f^\infty (R^2 z))\| \leq & c_4 (c, p, N, \mathcal{U}, \delta_0),\ \forall |z| \geq \frac{\delta}{2R},\\
	|\zeta^\infty (R^2 z)| \leq & \frac{c}{c_1 (c, p, N, \mathcal{U}, \delta_0)}\frac{1}{R^2 |z|}, \ \forall |z| \geq \frac{\delta}{2R},\\
	\left| \frac{\p \zeta^\infty(R^2 z)}{\p s} \right| =& \left| R^2 \frac{\p \zeta^\infty}{\p s} (R^2 z) \right| = \left| R^2 d \exp_y^{-1} (f^\infty(R^2 z)) \frac{\p f^\infty}{\p s} (R^2 z)\right|\\
	\leq & \frac{c}{c_1 (c, p, N, \mathcal{U}, \delta_0)} \frac{R^2}{1 + \frac{\delta^2 R^2}{4}},\ \forall |z| \geq \frac{\delta}{2R}\\
	\left| \frac{\p^2 \zeta^\infty(R^2 z)}{\p s^2}\right| =& \left| R^4 \frac{\p^2 \zeta^\infty}{\p s^2} (R^2 z) \right|\\
	=& \left| R^4 d^2 \exp^{-1}_y (f^\infty (R^2 z))\left(\frac{\p f^\infty}{\p s} (R^2 z)\right)^2 +\right.\\
	&\left. R^4 d\exp^{-1}_y (f^\infty (R^2 z)) \frac{\p^2 f^\infty}{\p s^2} (R^2 z) \right|\\
	\leq & c_4 (c, p, N, \mathcal{U}, \delta_0) (c^2 + c) \frac{R^4}{\left( 1 + \frac{\delta^2 R^2}{4}\right)^2},\ \forall |z| \geq \frac{\delta}{2R}
\end{align*}

We now divide $\mathcal{F}_{f^R} (0)$ into terms and estimate each term, eventually we will get the same estimate as in the case $\frac{1}{\delta R} \leq |z| \leq \frac{2}{\delta R}$.

Now we can take a look at what we should choose to be $\sigma$ and $c$ in the condition of the implicit function theorem. First, from the previous section, we have
\begin{align*}
	\| Q_{f^R} \|_{2,p,R} \leq c_0 (c, p, N, \mathcal{U})
\end{align*}
So we should choose $\tilde{c} = c_0 (c, p, N, \mathcal{U})$.

Since we want
\begin{align*}
	\| \mathcal{F}_{f^R} (0) \|_{0,p,R} < \frac{\epsilon}{4\tilde{c}}
\end{align*}
while we have proved
\begin{align*}
	\| \mathcal{F}_{f^R} (0) \|_{0,p,R} \leq C (c, p, N, \mathcal{U}, \delta_0) \frac{(\delta R + 1)}{(\delta R)^{\frac{2}{p}}} < C (c, p, N, \mathcal{U}, \delta_0) (1 + \delta_0) (\delta R)^{1 - \frac{2}{p}}
\end{align*}
where $C$ is increasing with respect to $\delta_0$.

We first choose $c_0$ and $\delta_0$ as in Proposition \ref{approx_right_inverse}. We know that the results of the Proposition still hold if we make $\delta_0$ smaller. Furthermore, we have seen that $C (c, p, N, \mathcal{U}, \delta_0)$ is increasing with respect to $\delta_0$.

For $\epsilon$, we first require $B_\epsilon (0, X) \subset U$. This upper bound for $\epsilon$ only depends on $c, p, N, \mathcal{U}$.

Then we can consider the $(0, p, R)$-norm of $\mathcal{F}_{f^R} (0)$ being smaller than $\epsilon / \tilde{c}$. By making $\delta_0$ small enough only depending on $c, p, N, \mathcal{U}$, we can make $\sigma$ as small as we like. Now we want to have
\begin{align*}
	\| \xi\| < \epsilon \Rightarrow \| d\mathcal{F}_{f^R} (\xi) - D_{f^R}\| \leq \frac{1}{2\tilde{c}}.
\end{align*}
For this, we apply Lemma \ref{difference}, which is an analog of Proposition 3.5.3 in McDuff and Salamon \cite[Proposition 3.5.3, p. 70]{mcduff}.

Recall the formula for $D_{f^R}$ from Equation (\ref{intrinsic}). For any $\xi \in W^{2, p} (\Sigma_1 \#_{\delta, R} \Sigma_2, (f^R)^{-1} TN)$, let $Z_\alpha = (f^R)_* e_\alpha$ for $\alpha = 1, 2$, where $\{e_1 ,e_2\}$ is a local orthonormal frame of $T(\Sigma_1 \#_{\delta, R} \Sigma_2)$.
\begin{align*}
	D_{f^R} (\xi) = \nabla^* \nabla \xi + \sum_\alpha R^N (Z_\alpha, \xi) Z_\alpha.
\end{align*}

Let us check the conditions of Theorem \ref{tailored_IFT}. The first condition is met in Section \ref{approx_right_inverse}, where we eventually constructed the bounded right inverse $Q_{f^R}$. We have also managed to satisfy the second and thrid conditions in the previous discussions in this section. Thus we can apply the implicit function theorem. In particular, we know there exists a unique $\xi$ such that $\exp_{f^R} (\xi)$ is a harmonic map. Although we are considering $S^2$s, everything can be done the same way for general Riemann surfaces. Combining the above, we arrive at the following theorem:

\begin{reptheorem}{existence_thm}[Existence of the Extended Gluing Map]
	For any $(\tilde{f}_1, \tilde{f}_2) \in \mathcal{M} (c, p)$, there exists a neighborhood $\mathcal{U}$ in $\mathcal{M} (c, p)$ and $\delta_0 = \delta_0 (c, p, \Sigma_1,  \Sigma_2, x_1, x_2, N, \mathcal{U})> 0$, such that for each pair of $(\delta, R) \in \mathcal{A} (\delta_0)$, there exists a gluing map $\imath_{\delta, R}: \mathcal{U} \rightarrow W^{2, p} (\Sigma_1 \#_{\delta, R} \Sigma_2, N) \times \tilde{V}$ such that each element $(\exp_{f^R} \xi, \tilde{v}) \in \imath_{\delta, R} (\mathcal{U})$ satisfies
	\begin{equation}
		\label{extended_harmonic_map_equation}
		\mathcal{F}_{f^R} (\xi) + v = 0
	\end{equation}
	where $\Sigma_1 \#_{\delta, R} \Sigma_2$ is the glued manifold as defined in \ref{weighted_metric}, and $\tilde{V}$ is defined in Definition \ref{V}, and $v = \sigma (\tilde{v})$, where $\sigma$ is defined in Definition \ref{sigma}. Furthermore, for any $\epsilon > 0$, we can choose $\delta_0 = \delta_0 (c, p, \Sigma_1,  \Sigma_2, x_1, x_2, N, \mathcal{U})$ such that, for any $(f_1, f_2) \in \mathcal{U}$, there exists $\xi \in W^{2, p} (\Sigma_1 \#_{\delta, R} \Sigma_2)$ satisfying
	\begin{equation*}
		\imath_{\delta, R} ((f_1, f_2)) = (\exp_{f^R} \xi, \tilde{v}), \quad \|(\xi^R, \tilde{v})\|_{2, p, R, V} < \epsilon
	\end{equation*}
	where $f^R$ denotes the pregluing of $f_1, f_2$ defined in (\ref{fR}), and the norm is defined in Definition \ref{weighted_with_V}.
	
	In particular, consider 
	\begin{equation*}
		\imath_{\delta, R} ((f_1, f_2)) \mid_{\tilde{V}} = 0.
	\end{equation*}
	If there are elements in $\imath_{\delta, R} (\mathcal{U})$ that satisfy the above equation, then these elements form a subset of the image of the gluing map consisting of harmonic maps. Otherwise, there is no harmonic map in the image of the gluing map.
\end{reptheorem}

\appendix
	
	\section{Choice of Cokernel Representatives}
	\label{choice_coker_rep}
	
	First, for $D_{f_i}$, $i = 1, 2$ defined in (\ref{Df}), we know that it is Fredholm from Lemma \ref{fredholm}. We would like to show that we can choose representatives of the quotient space $L^p_{f_i} / \text{Im} D_{f_i}$ such that, all representatives are supported away from an open neighborhood of $x_i$. To simplify notations, we omit the subscripts.
	
	\begin{lemma}
		\label{choice1}
		Consider $f: \Sigma \rightarrow N$ and let $D_f$ be the operator defined in (\ref{D_f}). Choose any $x\in \Sigma$. We can choose representatives $\{v_1, \cdots, v_k\}$ in $L^p_{f}$ that span the quotient space $L^p_{f} / \text{Im} D_{f}$, such that, $v_i$ is supported away from an open neighborhood of $x$ for all $i$.
	\end{lemma}
	\begin{proof}
		We claim that the conclusion of the lemma holds for any operator such that the target space is $L^p_{f}$, and the cokernel is finite dimensional, and the image is closed. We use mathematical induction to prove this.
		
		Suppose that the cokernel has dimension 0, there is nothing to prove.
		
		Suppose that the cokernel has dimension 1, we prove it by contradiction. Assume that for any $r > 0$, no representative of the quotient space vanishes on $B_r (x)$. Then we know that there exists $v$ as a representative of the quotient space that is only supported on $B_r (x)$. To see this, first choose any representative of the quotient space, which we denote by $w$. If $w$ is supported on $B_r (x)$, then let $v = w$. Otherwise, let $v$ be $w$ restricted to $B_r (x)$, that is, $v := w\mid_{B_r (x)}$. Since $w - v$ is supported away from $B_r (x)$, we know, by assumption, that $w - v$ must be in the operator's image. Thus, $v = w - (w - v)$ will be a representative of the quotient space supported on $B_r (x)$. Now take any $\eta$ supported outside $B_r (x)$, since no representative of the quotient space vanishes on $B_r (x)$, $\eta$ must be in the image. We can shrink $r$, and such $\eta$s can converge to any element in $L^p_f$. Since the image of the operator is closed, this means that the operator is surjective. Contradiction!
		
		Now suppose that the conclusion is true for cokernel dimension less than or equal to $k$ where $k \geq 1$. We prove the conclusion for cokernel dimension $k + 1$. First choose any representative $v$. $v$ may not be supported away from a neighborhood of $x$. Let $\sigma$ be the identity map of the space spanned by $v$. Write the original operator as $D$, then we know that the image of $D\oplus \sigma$ is closed (note that it is easy to show that a closed subspace direct sum with a one-dimensional space is closed). Now we can choose $v_1$ that is a representative of the cokernel of $D \oplus \sigma$ and supported away from a neighborhood of $x$. Let $\sigma_1$ be the identity map of the space spanned by $v_1$. Now consider $D\oplus \sigma_1$ and we can find other representatives that are supported away from a neighborhood of $x$.
	\end{proof}
	
	Now we prove that $D_{1, 2}$ defined in Definition \ref{spaces_and_operator} is Fredholm. Since $D_{f_1}$ and $D_{f_2}$ are Fredholm, it is obvious that $D_{1, 2}$ has finite-dimensional kernel. Thus we only need to prove that $D_{1, 2}$ has a finite-dimensional cokernel. The fact that $D_{1, 2}$ has closed range will follow from the fact that $D_{1, 2}$ has finite-dimensional kernel and finite-dimensional cokernel (refer to Abramovich and Aliprantis \cite[Section 2.1, Corollary 2.17, p. 76]{Abramovich2002}).
	\begin{lemma}
		\label{D12fredholm}
		$D_{1, 2}$ defined in Definition \ref{spaces_and_operator} is Fredholm.
	\end{lemma}
	\begin{proof}
		We only need to show that the cokernel is finite-dimensional.
		
		First, choose representatives $\{v_{i, 1}, \cdots, v_{i, k_i} \}$ in the quotient space $L^p_{f_i} / \text{Im} D_{f_i}$. Let $\sigma_i$ be the identity map of the finite-dimensional space spanned by these representatives. Consider $D_{1, 2} \oplus \sigma_1 \oplus \sigma_2$. Suppose the cokernel is not finite-dimensional, we can find $\{v_1, \cdots, v_n\}$, where $n = \dim T_y N + 1$ (recall $y := f_1 (x_1) = f_2 (x_2)$), such that each $v_i$ is a representative of the quotient map $(L^p_{f_1} \times \L^p_{f_2}) / \text{Im} (D_{1, 2} \oplus \sigma_1 \oplus \sigma_2$, and we can let them be linearly independent in the quotient space. Consider $v_i (x_1) - v_i (x_2)$, since there are $\dim T_y N + 1$ elements, we know there is some nontrivial linear combination of the elements, which we denote by $\tilde{v}$, such that $\tilde{v} (x_1) - \tilde{v} (x_2) = 0$. However, that means $\tilde{v}$ is in the image of $D_{1, 2} \oplus \sigma_1 \oplus \sigma_2$. Contradiction!
	\end{proof}
	
	From the proof of Lemma \ref{choice1}, we know we are only using certain properties of the operators. Thus, it is easy to use the same proof to get the following lemma.
	\begin{lemma}
		\label{choice2}
		We can choose representatives $\{v_1, \cdots, v_k\}$ in $L^p_{f_1} \times L^p_{f_2}$ that span the quotient space $(L^p_{f_1} \times L^p_{f_2}) / \text{Im} D_{1, 2}$, such that, $v_i$ is supported away from some open neighborhoods of $x_1$ and $x_2$.
	\end{lemma}

	\section{Uniform Boundedness of Coordinate Change}
	\label{uniform_bound_proof}
	
	Consider a smooth Riemannian manifold $N$ and a point $p_0 \in N$. By the uniformly normal neighborhood lemma, there exists a neighborhood $U$ containing $p_0$ and $\delta > 0$ such that
	\begin{enumerate}
		\item For all $p_1, p_2 \in U$, there exists a unique geodesic $\gamma$ of length less than $\delta$ joining $p_1$ to $p_2$. Moreover, $\gamma$ is minimizing.
		\item For any $p \in U$, $U \subset \exp_p (B_{p} (\delta))$, and $\exp_p$ is a diffeomorphism on $B_p (\delta)$.
	\end{enumerate}
	
	For such a uniformly normal neighborhood $U$ of $p_0$ and any point $p_1 \in U$, there exists a unique smooth geodesic $\gamma$ such that $\gamma(0) = p_0$ and $\gamma(1) = p_1$.
	
	Consider a fixed normal coordinate chart centered at $p_0$. Denote the corresponding coordinates of $p_1$ by $y_1$. Note that $y_1$ is equal to the coordinates of $\dot{\gamma} (0)$. Consider a vector $V \in T_{p_0} N$ and let $V(t)$ for $t \in [0, 1]$ be the parallel transport along $\gamma$. We have the following system of ODEs:
	\begin{align*}
		&\ddot{\gamma}^k (t) + \Gamma^k_{ij} (\gamma(t)) \dot{\gamma}^i (t) \dot{\gamma}^j (t) = 0,\\
		&\dot{V}^k (t) + \Gamma^k_{ij} (\gamma (t)) V^i (t) V^j (t) = 0.
	\end{align*}
	By setting $X(t) := \dot{\gamma} (t)$, we can transform the above system into a system of first order ODEs. Thus we know $V(1)$ will be a smooth function of $V(0)$, $\gamma(0)$, and $\dot{\gamma} (0)$. Here, $\gamma(0) \equiv p_0$. Denote the parallel transport by
	\begin{equation}
		\label{parallel_transport}
		PT: T_{p_0} N \times T_{p_0} N \rightarrow T_p N
	\end{equation}
	In coordinates (normal coordinates at $p_0$ we fixed above) we have:
	\begin{alignat*}{2}
		PT:  &\mathbb{R}^n \times \mathbb{R}^n &&\rightarrow \mathbb{R}^n\\
		&(V(0), \dot{\gamma}(0)) &&\mapsto V(1)
	\end{alignat*}
	Equivalently, we can write $PT$ as
	\begin{alignat*}{2}
		PT:  &\mathbb{R}^n &&\rightarrow \mathbb{R}^n \times \mathbb{R}^n\\
		&\dot{\gamma}(0) &&\mapsto \text{matrix representing the parallel transport}
	\end{alignat*}
	Note that for the matrix, the $k$-th column is simply $PT (e_k, \dot{\gamma})$ written in coordinates, so the matrix is a smooth function of $\dot{\gamma}(0)$. Also note that $PT (0)$ is the identity matrix. We can write $PT (w) = id + M(w)$ where $M$ is some matrix whose norm goes to $0$ uniformly as $|w|$ goes to $0$. Similarly, we can control $PT^{-1}$.

	\section{Perturbations in the Pregluing}
	
	First, let us identify the Sobolev spaces regarding the perturbed map with the space regarding the original map. Namely, we identify $W^{2, p} (M, (f^0)^{-1} TN)$\\ (resp. $L^p (M, (f^0)^{-1} TN)$) with $W^{2, p} (M, (f^{0, r})^{-1} TN)$ (resp. $L^p (M, (f^{0, r})^{-1} TN)$). For simplicity, we only write out the case for $f^0$. The case for $f^\infty$ is completely the same.
	
	\begin{proposition}[Equivalence of Sobolev Spaces under Perturbations]
		\label{space_equivalence}
		For any\\ $(\tilde{f}^0, \tilde{f}^\infty) \in \mathcal{M} (c, p)$, where $\mathcal{M} (c, p)$ is defined as in Definition \ref{moduli_defn_nonsurj}, we can choose a neighborhood $\mathcal{U}$ in $\mathcal{M} (c, p)$ and $\delta_0 = \delta_0 (c, p, N, \mathcal{U})$ such that, for any $(f^0, f^\infty) \in \mathcal{U}$ and $(\delta, R) \in \mathcal{A} (\delta_0)$, we have:
		
		(\romannum{1}) $L^p (M, (f^0)^{-1} TN) \cong L^p (M, (f^{0, r})^{-1} TN)$.
		
		(\romannum{2}) $W^{2, p} (M, (f^0)^{-1} TN) \cong W^{2, p} (M, (f^{0, r})^{-1} TN)$.
		
		where the constants in the equivalence relation only depend on $c, p, N, \mathcal{U}$.
	\end{proposition}
	
	\begin{proof}
		Recall $y = f^0 (0) = f^\infty (\infty)$. For $|z| \geq \frac{2}{\delta R}$, we have $f^0 = f^{0, r}$. For $|z| \leq \frac{2}{\delta R}$, $f^0 (z) = \exp_y \zeta^0 (z)$ whereas $f^{0, r} (z) = \exp_y (\rho(\delta Rz) \zeta^0 (z))$. For any $\xi \in W^{2, p} (M, (f^0)^{-1} TN)$, for $|z| \leq \frac{2}{\delta R}$, we can use parallel transport from $\exp_y \zeta^0(z)$ to $\exp_y (\rho(\delta Rz) \zeta^0 (z))$ to get $\tilde{\xi}$ (For $|z| \geq \frac{2}{\delta R}$ we can just take $\tilde{\xi} = \xi$). In other words:
		\begin{equation*}
			\tilde{\xi} (z) := P_\gamma \xi(z)
		\end{equation*}
		where $P_\gamma$ is parallel transport along $\gamma$, $\gamma: [0, 1] \rightarrow N$ is a geodesic from $\gamma(0) = f^0 (z) = \exp_y \zeta^0(z)$ to $\gamma(1) = f^{0, r} (z) = \exp_y (\rho(\delta Rz)\zeta^0(z))$. Since $\|df^0\|$ is bounded by $c$, for small enough $\delta_0$, there always exists such a unique geodesic $\gamma$, whose image is a subset of $\exp_y (t\zeta^0(z))$.
		
		Now we have defined the map, we have to show that the map will induce the isomorphisms. First we have to show that this map maps $L^p$ (resp. $W^{2,p}$) to $L^p$ (resp. $W^{2,p}$) functions, then we have to show that after identifying the two spaces, the two norms are equivalent. We know the map is identity if we only consider $|z| \geq \frac{2}{\delta R}$. Recall that for the norms, we consider a finite cover of $M$ such that on each set of the cover there is a coordinate chart for $M$ and the image of that set has a coordinate chart on $N$. We can choose a cover such that there is a coordinate chart in the cover that contains $0 \in S^2$. For small enough $\delta_0$, this chart will contain all $|z| \leq \frac{4}{\delta R}$. We can choose the cover such that all other charts in this cover will only contain $|z| \geq \frac{2}{\delta R}$. Then we only have to consider the chart containing $0$. We also make the $\delta_0$ small enough so that, for this chart, we can choose the coordinates on $N$ to be the normal coordinates at $y$.
		
		Now we can only consider the chart centered at $y$. First, we identify the $L^p$ and $W^{2, p}$ spaces (both for $f^0$ and $f^{0, r}$) with $L^p (M, T_y N)$ and $W^{2, p} (M, T_y N)$. We construct the map by parallel transport along the geodesic ending at $y$. Recall the parallel transport map $PT$ defined in \ref{parallel_transport}. Since we can write $PT(w) = id + S(w)$ where the norm of $S$ goes to $0$ as $w\rightarrow 0$, we have
		\begin{equation*}
			L^p (M, (f^0)^{-1} TN) \cong L^p (M, T_y N) \cong L^p (M, (f^{0, r})^{-1} TN)
		\end{equation*}
		for $\delta_0$ small enough.
		
		Next, we consider the case for $W^{2, p}$. We can consider a fixed cutoff function on this chart such that the function is equal to $1$ for all $|z| \leq \frac{2}{\delta R}$ for $\delta_0$ small enough. The section multiplied by this cutoff function together with the terms on other charts will be an equivalent norm. This way, we can use approximation by compactly supported smooth sections for $W^{2, p}$. For a smooth section with compact support $\phi$,
		\begin{align*}
			\frac{\p}{\p x^\alpha} PT (\phi) = \frac{\p}{\p x^\alpha} \phi + \frac{\p}{\p x^\alpha} (S\phi) = PT \left(\frac{\p \phi}{\p x^\alpha}\right) + \frac{\p S}{\p x^\alpha} \phi
		\end{align*}
		Note that $\frac{\p S}{\p x^\alpha} = \frac{\p S}{\p y^i}\frac{\p f^i}{\p x^\alpha}$, where $f$ can be $f^0$ or $f^{0, r}$. we want to show that the norm of $\frac{\p S}{\p x^\alpha}$ is bounded by some constant that does not depend on $\delta$ or $R$. This is obviously true for $f^0$. For $f^{0, r} (z) = \exp_y (\rho(\delta Rz) \zeta^0(z))$, we want to show the norm of $\frac{\p f^{0, r}}{\p x^\alpha}$ is bounded:
		\begin{align*}
			\left| \frac{\p S}{\p x^\alpha} \right| &= \left| \frac{\p S}{\p y^i} \frac{\p f^i}{\p x^\alpha} \right| \\
			&= \left| \frac{\p S}{\p y^i} \frac{\p \exp_y^i (\rho(\delta Rz) \zeta^0 (z))}{\p x^\alpha} \right| \\
			&\leq \left| \frac{\p S}{\p y^i}\right| \left| d\exp_y (\rho(\delta Rz) \zeta^0 (z)) \right| \left| \delta R \frac{\p \rho}{\p x^\alpha} (\delta Rz) \zeta^0(z) + \rho(\delta Rz) \frac{\p \zeta^0(z)}{\p x^\alpha}\right|
		\end{align*}
		By definition of $\zeta^0$, we have
		\begin{align*}
			f^0(z) = \exp_y (\zeta^0(z))
		\end{align*}
		Since the $C^0$ norm of $\zeta^0 (z)$ is bounded by  some constant times $1/\delta R$, we know we can control the $C^0$ norm of $\frac{\p S}{\p x^\alpha}$.
		
		For the second order derivatives, we have
		\begin{align*}
			&\frac{\p^2 }{\p x^\alpha \p x^\beta} PT (\phi) = PT \left(\frac{\p^2 }{\p x^\alpha \p x^\beta} \phi \right) + \left(\frac{\p^2 }{\p x^\alpha \p x^\beta} S\right) \phi + \frac{\p}{\p x^\alpha} S \frac{\p}{\p x^\beta} \phi +  \frac{\p}{\p x^\beta} S \frac{\p}{\p x^\alpha} \phi
		\end{align*}
		We only need to show that $\|\frac{\p^2 M}{\p x^\alpha \p x^\beta} \phi\|_{L^p} \leq C \|\phi\|_{W^{2, p}}$, where $C$ only depends on $\delta_0$. Again, this is obviously true for $f^0$, so we only need to consider $f^{0, r}$. We write out $\frac{\p^2 S}{\p x^\alpha \p x^\beta}$:
		\begin{align*}
			&\left| \frac{\p^2 S}{\p x^\alpha \p x^\beta} \right| = \left| \frac{\p^2 S}{\p y^i \p y^j} \frac{\p f^i}{\p x^\alpha}\frac{\p f^j}{\p x^\beta} + \frac{\p S}{\p y^i} \frac{\p^2 f^i}{\p x^\alpha \p x^\beta} \right|
		\end{align*}
		The only term that matters is $\frac{\p^2 f^i}{\p x^\alpha \p x^\beta}$. We can get by direct computation:
		\begin{align*}
			\left| \frac{\p^2 \exp_y^i (\rho(\delta Rz) \zeta^0(z))}{\p x^\alpha \p x^\beta} \right| \leq & C(\delta_0) (\delta R + 1)
		\end{align*}
		Then we have
		\begin{align*}
			\left\|\frac{\p^2 S}{\p x^\alpha \p x^\beta} \phi\right\|_{L^p} \leq C(\delta_0) \|\phi\|_{C^0} \leq C(\delta_0) \|\phi\|_{W^{2, p}}
		\end{align*}
		by the Sobolev embedding.
		
		So far we have shown that $\|\tilde{\xi}\|_{W^{2, p}} \leq \|\xi\|_{W^{2, p}}$. The proof for $f^\infty$ is the same. For the other direction, consider $P^{-1} = (id + S)^{-1}$. For $\delta_0$ small enough so that the entries of $S$ are sufficiently small, $(id + S)^{-1} $ exists and is just a function of the entries of $S$. Thus we have finished the proof.
	\end{proof}
	
	\begin{proposition}
		\label{D_perturbation}
		For any $(\tilde{f}^0, \tilde{f}^\infty) \in \mathcal{M} (c, p)$ and any $\epsilon > 0$, we can choose a neighborhood $\mathcal{U}$ in $\mathcal{M} (c, p)$ and $\delta_0 = \delta_0 (c, p, N, \mathcal{U}, \epsilon) > 0$ such that, for any $(\delta, R) \in \mathcal{A} (\delta_0)$ and any $(f^0 , f^\infty) \in \mathcal{U}$, under the identification in Proposition \ref{space_equivalence},
		\begin{equation*}
			\|D_{0, r} - D_0\| < \epsilon,
		\end{equation*}
		where $D_{0, r} := D_{f^{0, r}}$ and $D_0 := D_{f^0}$. Recall the definition of $D_f$ in Equation (\ref{Df}) and the definition of $f^{0, r}$ in Equation (\ref{f0r}).
	\end{proposition}
	\begin{proof}
		Still, we only have to consider the chart centered at $y$ and $|z| \leq \frac{2}{\delta R}$. As in the proof of Proposition \ref{space_equivalence}, let $\tilde{\xi}$ and $\xi$ be sections of $W^{2, p} (M, (f^0)^{-1}TN)$ and $W^{2, p} (M, (f^0)^{-1} TN)$ respectively such that they are the same after the identification. We want to prove that for $\delta_0$ small enough, $\| D_{0, r} \tilde{\xi} - D_0 \xi\|_{L^p} <{ \epsilon} \|\xi\|_{W^{2, p}}$ for any such pair of sections. We have
		\begin{equation*}
			\| D_{0, r} \tilde{\xi} - D_0 \xi \|_{L^p} \leq \| D_{0, r} \tilde{\xi} - D_{0, r} \xi \|_{L^p} + \| D_{0, r} \xi - D_0 \xi \|_{L^p} 
		\end{equation*}
		Consider the right hand side. For the first term, by (\ref{Df}) and Appendix \ref{uniform_bound_proof}, we only need to control $\| \tilde{\xi} - \xi\|_{W^{2, p}}$. In the proof of Proposition \ref{space_equivalence}, we have $\tilde{\xi}^i = \xi^k (\delta_i^k + M_{ik})$, where the norm of $M$ converges to 0 as $\delta_0$ converges to 0. So the second order derivative term in $\tilde{\xi} - \xi$ can be controlled. For the first order derivative term, we can use the Sobolev embedding to control the $L^{p^*}$ norm (of first order derivative), where $p^* = 2p / (2 - p) > 2p$. Then consider its multiplication with the characteristic function of the neck, and use Holder's inequality. Since the measure of the neck converges to 0, we can control the norm. The second term can be controlled similarly.
	\end{proof}
	
	\section{Apriori Estimates for the Differential Operators}
	\label{apriori}
	
	First, we write out the intrinsic formula for the linearization of the harmonic map operator and its proof by T. Parker.
	
	Let $F: M \times (-\epsilon, \epsilon) \times (-\epsilon, \epsilon)$ be a two-parameter variation of a map $f: M\rightarrow N$ (not assumed to be harmonic). Write $F (x, s, t)$ as $f_{s, t} (x)$, so $f = f_{0, 0}$ and set
	\begin{equation}
		\label{Zalpha}
		X = f_* \frac{\p}{\p s} \quad Y = f_* \frac{\p}{\p t} \quad Z_\alpha = f_* e_\alpha
	\end{equation}
	where $\{e_1, e_2\}$ is a local orthonormal frame of $TM$. The energy of $f_{s, t}$
	\begin{equation*}
		E(f_{s, t}) = \frac{1}{2} \int_M |df_{s, t}|^2
	\end{equation*}
	is a function of $(s, t)$. The first variation of $E$, applied to $X$, is the function
	\begin{equation*}
		\phi(s, t) = (\delta E)_{f_{s, t}} (X) = \int_M \langle P(f_{s, t}), X\rangle,
	\end{equation*}
	where $P(f) = \nabla^* df$. The partial derivative of $\p_t \phi (s, t)$ with respect to $t$ is, on the one hand,
	\begin{equation}
		\int_M Y \cdot \langle P(f_{s, t}), X\rangle = \int \langle \nabla_Y P(f_{s, t}), X\rangle + \langle P(f_{s, t}), \nabla_Y X\rangle.
	\end{equation}
	On the other hand, $\p_t \phi(s, t)$ is the second variation of $E$, which is given by the standard formula
	\begin{align*}
		(\delta^2 E)_{f_{s, t}} (X, Y) &= \int_M \langle \nabla Y, \nabla X\rangle - \sum_\alpha \langle f_* (e_\alpha), R^N (f_* (e_\alpha), Y) X\rangle + \int_M \langle \nabla^* df, \nabla_X Y\rangle\\
		&= \int_M \langle \nabla^* \nabla Y, X\rangle + \sum_\alpha \langle R^N (Z_\alpha, Y) Z_\alpha, X\rangle + \int_M \langle P(f), \nabla_X Y\rangle
	\end{align*}
	(cf. J. Jost \cite[Theorem 8.2.1]{jost}). Comparing the last two displayed equations shows that
	\begin{equation*}
		\nabla_Y P = \nabla^* \nabla Y + \sum_\alpha R^N (Z_\alpha, Y) Z_\alpha.
	\end{equation*}
	Using the setup and notation of equations (\ref{222})-(\ref{Df}), $D_f (Y) = \nabla_Y P(f_{s, t})\mid_{s = t = 0}$. Thus
	\begin{equation}
		\label{intrinsic}
		D_f (Y) = \nabla^* \nabla Y + \sum_\alpha R^N (Z_\alpha, Y) Z_\alpha.
	\end{equation}

	We will prove an $L^p$ estimate for the operators $D_0 := D_{f^0}$ and $D_{f^\infty} := D_\infty$. Recall the definition of $D_f$ in Equation (\ref{Df}).
	
	\begin{proposition}
		Let $f$ be a smooth map from $M$ to $N$, where $M$ and $N$ are Riemannian manifolds. Then $D_f$ defined in (\ref{Df}) is locally a second order strongly elliptic system.
	\end{proposition}
	
	\begin{proof}
		Say we use the normal coordinates centered at $f(x_0)$ on $N$. From (\ref{D_f}) we get operator in coordinates, we then consider the coordinate change and apply Lemma \ref{coord_est}, we get the result.
	\end{proof}
	
	\begin{lemma}[$L^p$ estimate]
		\label{L^p_estimate}
		For any $(\tilde{f}^0, \tilde{f}^\infty) \in \mathcal{M} (c, p)$ where $\mathcal{M} (c, p)$ is as defined in Definition \ref{moduli_defn_nonsurj}, we can choose a neighborhood $\mathcal{U}$ in $\mathcal{M} (c, p)$ such that, for every $(f^0, f^\infty) \in \mathcal{U}$ and $\xi = (\xi^0, \xi^\infty) \in W^{2,p}_{f^{0,\infty}}$, we have
		\begin{align*}
			\|\xi\|_{W^{2,p}} \leq c_0 (\|D_{0,\infty} \xi\|_{L^p} + \|\xi\|_{L^p})
		\end{align*}
		where $c_0$ only depends on $c, p, N, \mathcal{U}$.
	\end{lemma}
	
	\begin{proof}
		We only need to show that
		\begin{equation*}
			\|\xi^0\|_{W^{2,p}} \leq c_0 (\|D_{f^0} \xi^0\|_{L^p} + \|\xi^0\|_{L^p}).
		\end{equation*}
		From (\ref{D_f_same_chart}) we know
		$D_0 : W^{2,p}(S^2, {(f^0)}^{-1} TN) \rightarrow L^p (S^2, {(f^0)}^{-1} TN)$ is a linear elliptic system, where each second order derivative term only contains a single component. We can apply the proof for $L^p$ estimates in \cite{gilbarg_trudinger} (See section 9.5) to prove the same for our case of elliptic system.
	\end{proof}
	
	The following is an analog of Lemma 10.6.1 in \cite{mcduff}:
	\begin{lemma}(Compare McDuff and Salamon \cite[Lemma 10.6.1, p. 392]{mcduff} for the analogous statement)
		\label{Q_est}
		For any $(\tilde{f}^0, \tilde{f}^\infty) \in \mathcal{M} (c, p)$, where $\mathcal{M} (c, p)$ is defined as in Definition \ref{moduli_defn_nonsurj}, we can choose a neighborhood $\mathcal{U}$ in $\mathcal{M} (c, p)$ and positive constants $\delta_0$ and $c_0$ only depending on $c, p, N, \mathcal{U}$ such that, for all $(f^0, f^\infty) \in \mathcal{U}$ and $(\delta, R) \in \mathcal{A} (\delta_0)$, the following holds for $r := \delta R$:
		
		(\romannum{1}) For every $\xi = (\xi^0 , \xi^\infty) \in W^{2,p}_{f^{0,\infty,r}}$, we have
		\begin{equation*}
			\|\xi\|_{W^{2,p}} \leq c_0 (\| D_{0,\infty,r} \xi\|_{L^p} + \|\xi\|_{L^p}).
		\end{equation*}
		where $W^{2, p}_{f^{0, \infty, r}} := W^{2, p}_{f^{0, r}, f^{\infty, r}}$ and $D_{0, \infty, r} := D_{f^{0, r}, f^{\infty, r}}$. Recall the definition of $W^{2, p}_{f_1, f_2}$ in Definition \ref{spaces_and_operator}.
		
		(\romannum{2}) For every $(\xi, \tilde{v}) = (\xi^0 , \xi^\infty, \tilde{v}) \in W^{2,p}_{f^{0,\infty,r}} \times \tilde{V}$, we have
		\begin{equation*}
			D_{0,\infty,r} \xi + \sigma(\tilde{v}) = 0\ \Rightarrow\ \|\xi\|_{W^{2,p}} + \|\tilde{v}\|_{\tilde{V}} \leq c_0 \|(\xi, \tilde{v})\|_{L^2}.
		\end{equation*}
		
		(\romannum{3}) For every $\eta = (\eta^0, \eta^\infty) \in L^p_{f^{0,r}} \times L^p_{f^{\infty,r}}$ we have
		\begin{equation*}
			\| Q_{0,\infty,r} \eta \|_{W^{2,p} \times \tilde{V}} \leq c_0 \|\eta\|_{L^p},
		\end{equation*}
		where $Q_{0, \infty, r}$ is defined as in Equation (\ref{Q0infr}).
	\end{lemma}
	\begin{proof}
		The proof follows the same ideas as in McDuff and Salamon \cite[Lemma 10.6.1, p. 392]{mcduff}. Note that the part for $f^0$ and the part for $f^\infty$ are symmetric, so we only need to consider one, and the other can be proved in the same way. Let us consider $f^0$. For any $\xi^0 \in W^{2,p}_{f^{0,r}}$, note that $f^{0,r}$ only differs from $f^0$ in $|z| \leq \frac{2}{\delta R}$. In $|z| \leq \frac{2}{\delta R}$, we know that $f^0$ and $f^{0,r}$ are both close to $y$, so we can do a parallel transport of $\xi^0$ to get $\tilde{\xi^0} \in W^{2,p}_{f^0}$. From Lemma \ref{L^p_estimate} we know that
		\begin{align*}
			\|\tilde{\xi^0}\|_{W^{2,p}} \leq c_0 (\| D_{0} \tilde{\xi^0}\|_{L^p} + \|\tilde{\xi^0}\|_{L^p})
		\end{align*}
		From the previous section we know that the first estimate is proved.
		
		To prove the remaining estimates we consider the following abstract functional analytic setting. Suppose we have a surjective Fredholm operator
		\begin{equation*}
			D: \mathcal{W} \rightarrow \mathcal{L}
		\end{equation*}
		of index $d$ between two Banach spaces (think of the case $\mathcal{W} = W^{2,p}_{f^{0, \infty}} \times \tilde{V}$, and $\mathcal{L} = L^p_{f^0} \times L^p_{f^\infty}$, and $D = D_{0, \infty} \oplus \sigma$). We assume $\mathcal{W}$ is equipped with an inner product $\langle\cdot,\cdot\rangle$ and denote the corresponding norm by
		\begin{equation*}
			\|\xi\|_{L^2} := \sqrt{\langle \xi, \xi\rangle}.
		\end{equation*}
		(think of the $L^2$ inner product on $W^{2, p}_{f^{0, \infty}}$ and let $\tilde{V}$ be orthogonal to $W^{2, p}_{f^{0, \infty}}$) We assume further that there are positive constants $c$ and $c_D$ such that
		\begin{equation}
			\|\xi\|_{L^2} \leq c\|\xi\|_{\mathcal{W}}
		\end{equation}
		for every $\xi \in \mathcal{W}$ (this is true because the $L^2$ norm of the $W^{2, p}_{f^{0, \infty}}$ component is controlled by the $W^{2, p}_{f^{0, \infty}}$ norm) and
		\begin{equation}
			D\xi = 0 \ \Rightarrow\ \|\xi\|_{\mathcal{W}} \leq c_D \|\xi\|_{L^2}.
		\end{equation}
		This holds in our setting because the kernel of $D$ is finite dimensional. Denote by $Q : \mathcal{L} \rightarrow \mathcal{W}$ the right inverse of $D$ whose image is the orthogonal complement of the kernel with respect to the above inner product. Note that here $W^{2, p}_{f^{0, \infty}} \times \tilde{V}$ is just a subspace of the Hilbert space $L^2$. We can consider $\ker D$, which is finite dimensional, and its orthogonal space in $\mathcal{W}$, which we denote by $(\ker D)^\perp$. It is easy to show that $(\ker D)^\perp$ is closed (in $\mathcal{W}$), $\ker D \cap (\ker D)^\perp = \{0\}$, and $\ker D + (\ker D)^\perp = \mathcal{W}$. Thus we have $\mathcal{W} = \ker D \oplus (\ker D)^\perp$. Then we can construct the $Q$ the same way as for Hilbert spaces.
		
		Now we prove the norm of $Q$ is bounded: Since we already have $\mathcal{W} = \ker D \oplus (\ker D)^\perp$, we can consider $D$ on $(\ker D)^\perp$ and we know the inverse will be bounded by the open mapping theorem.
		
		Now suppose that $D^\prime : \mathcal{W} \rightarrow \mathcal{L}$ is another bounded linear operator ($D_{0,\infty,r} \oplus \sigma$ for example) such that
		\begin{equation*}
			\epsilon := \| D^\prime - D \| \|Q\| < 1.
		\end{equation*}
		Since $DQ = id$ we have $\| D^\prime Q - id\| < 1$ and so $D^\prime$ is surjective with right inverse $Q(D^\prime Q)^{-1}$. However, we wish to understand the right inverse $Q^\prime : \mathcal{L} \rightarrow \mathcal{W}$ whose image is the orthogonal complement of the kernel of $D^\prime$.
		
		As a first step we observe that, if $D^\prime \zeta = 0$, then $\| QD \zeta\|_{\mathcal{W}} \leq \epsilon\|\zeta\|_{\mathcal{W}}$ and $\|\zeta - QD\zeta\|_{\mathcal{W}} \geq (1 - \epsilon)\|\zeta\|_{\mathcal{W}}$. Hence
		\begin{equation}
			\label{17}
			D^\prime \zeta = 0\ \Rightarrow\ \|QD\zeta\|_{\mathcal{W}} \leq \frac{\epsilon}{1 - \epsilon} \|\zeta - QD\zeta\|_{\mathcal{W}}.
		\end{equation}
		Since $\zeta - QD\zeta \in \ker D$, we find $\| \zeta - QD\zeta\|_{\mathcal{W}} \leq c_D \|\zeta - QD\zeta\|_{\mathcal{L}} \leq c_D \|\zeta\|_{\mathcal{L}}$. The last inequality holds because $\zeta - QD\zeta$ is the orthogonal projection of $\zeta$ onto the kernel of $D$ (note that the image of $Q$ is $(\ker D)^\perp$.) But we saw above that $\|(1 - \epsilon)\zeta\|_{\mathcal{W}} \leq \|\zeta - QD\zeta\|_{\mathcal{W}}$. Hence
		\begin{equation*}
			D^\prime \zeta = 0\ \Rightarrow\ \|\zeta\|_{\mathcal{W}} \leq \frac{c_D}{1 - \epsilon}\|\zeta\|_{L^2}.
		\end{equation*}
		This proves (\romannum{2}).
		
		Now assume that $\xi \in \mathcal{W}$ is orthogonal to the kernel of $D^\prime$. Let $e_1, \cdots, e_d$ be an orthonormal basis of $\ker D$ and consider the basis $e_1^\prime, \cdots, e_d^\prime$ of $\ker D^\prime$ defined by
		\begin{equation*}
			e_i^\prime - QDe_i^\prime = e_i,\ i = 1,\cdots,d.
		\end{equation*}
		(The map $\ker D^\prime \rightarrow \ker D : \zeta \mapsto \zeta - QD\zeta$ is an isomorphism between the two kernels. We can see that since in the above argument, we have, if $D^\prime \zeta = 0$, then $\| QD \zeta\|_{\mathcal{W}} \leq \epsilon\|\zeta\|_{\mathcal{W}}$ and $\|\zeta - QD\zeta\|_{\mathcal{W}} \geq (1 - \epsilon)\|\zeta\|_{\mathcal{W}}$. Note that the index of Fredholm operators stays the same if we have $\|D^\prime - D\| < \epsilon$ for fixed $D$ and small enough $\epsilon$. Since the index is the same for $D^\prime$ and $D$, we know the map is surjective.)
		
		Since $\xi - QD \xi \in \ker D$ we have
		\begin{equation*}
			\xi - QD\xi = \sum_{i=1}^d \langle \xi,e_i\rangle e_i.
		\end{equation*}
		Moreover, $\langle \xi, e_i\rangle = \langle \xi, e_i - e_i^\prime \rangle = \langle \xi, -QDe_i^\prime\rangle$ and hence, by (\ref{17}),
		\begin{equation*}
			|\langle \xi, e_i \rangle| \leq c \|QDe_i^\prime\|_{\mathcal{W}} \|\xi\|_{\mathcal{W}} \leq \frac{c\epsilon}{1 - \epsilon} \|e_i\|_{\mathcal{W}} \|\xi\|_{\mathcal{W}} \leq \frac{cc_D \epsilon}{1 - \epsilon} \|\xi\|_{\mathcal{W}}.
		\end{equation*}
		Combining this with the previous identity we find
		\begin{equation*}
			\|\xi - QD\xi\|_{\mathcal{W}} \leq c_D \|\xi - QD\xi\|_{L^2} = c_D \sqrt{\sum_{i = 1}^d \langle \xi, e_i\rangle^2} \leq \frac{\sqrt{d}cc_D^2 \epsilon}{1 - \epsilon} \|\xi\|_{\mathcal{W}}.
		\end{equation*}
		Hence
		\begin{align*}
			\|\xi\|_{\mathcal{W}} &\leq \|QD\xi\|_{\mathcal{W}} + \|\xi - QD\xi\|_{\mathcal{W}}\\
			&\leq \|Q\| \|D^\prime \xi\|_{\mathcal{L}} + \|Q\| \|D^\prime - D\| \|\xi\|_{\mathcal{W}} + \|\xi - QD\xi\|_{\mathcal{W}}\\
			&\leq \|Q\| \|D^\prime \xi\|_{\mathcal{L}} + \epsilon \left( 1 + \frac{\sqrt{d}cc_D^2}{1 - \epsilon}\right) \|\xi\|_{\mathcal{W}}.
		\end{align*}
		If $\epsilon \leq 1/2$ and $\epsilon(1 + 2\sqrt{d}cc_D^2) \leq 1/2$ we deduce that $\|\xi\|_{\mathcal{W}} \leq 2\|Q\|\|D^\prime \xi\|_{\mathcal{L}}$ for every $\xi \in \mathcal{W}$ that is orthogonal to the kernel of $D^\prime$.
		
		Now recall that $\|D^\prime - D\| = \epsilon/\|Q\|$ and that $D^\prime$ is surjective, so that $D^\prime \xi$ runs over all elements in $\mathcal{L}$. It follows that there is a constant $\delta > 0 $ such that
		\begin{equation*}
			\|D^\prime - D\| < \delta\ \Rightarrow\ \|Q^\prime\| \leq 2 \|Q\|.
		\end{equation*}
		How small $\delta$ must be chosen depends only on the operator norm of $Q$ and the constants $d, c, c_D$. This finishes the proof.
	\end{proof}
	
	\section{Sobolev Inequalities}
	
	We consider Sobolev embeddings on the weighted sphere. We show that the constant in the inequality does not depend on $\delta$ or $R$, which will be needed in certain proofs. The following lemma is an analog of Lemma 10.3.1 in McDuff and \cite{mcduff}.
	
	\begin{lemma} (Compare McDuff and Salamon \cite[Lemma 10.3.1, p. 377]{mcduff} for the analogous statement for J-holomorphic curves)
		\label{sobolev_embedding}
		For any $(\tilde{f}^0, \tilde{f}^\infty) \in \mathcal{M} (c, p)$, where $\mathcal{M} (c, p)$ is defined as in Definition \ref{moduli_defn_nonsurj}, we can choose a neighborhood $\mathcal{U}$ in $\mathcal{M} (c, p)$, and $\delta_0 = \delta_0 (c, p, N, \mathcal{U})$ such that, for any $(f^0, f^\infty) \in \mathcal{U}$, and any $(\delta, R) \in \mathcal{A} (\delta_0)$, we have
		\begin{equation*}
			\| \xi \|_{L^\infty} \leq C (c, p, N, \mathcal{U}) \|\xi\|_{2, p, R}
		\end{equation*}
		for any $\xi \in W^{2, p} (S^2, (f^R)^{-1} TN)$, where $S^2$ is the weighted sphere defined in Equation (\ref{weighted_metric}) and $f^R$ is the pregluing defined in Equation (\ref{fR}).
	\end{lemma}
	
	\begin{proof}
		We let $\delta_0$ be small enough as in Section \ref{pregluing}. 
		
		For $|z| \geq 1 / R$, the metric is the Fubini-Study metric on $S^2 \backslash B_{1 / R}$. Each point $z_0 \in S^2 \backslash B_{1 / R}$ is contained in a disc $D$ of radius $\pi / 4$. Then we get the result from the usual Sobolev embedding.
		
		For $|z| \leq 1 / R$, we can consider the coordinate change $w = 1 / R^2 z$ and the result follows in the same way as above.
	\end{proof}

\section{Estimate for the Differential of the Harmonic Map Operator}

In this section, we prove a lemma that will be useful in the application of the implicit function theorem. The lemma and the proof are similar to those in McDuff and Salamon \cite[Proposition 3.5.3, p. 70]{mcduff}. Before stating this lemma, we make the following definition, which is an analog of Remark 3.5.1 in McDuff and Salamon \cite[Remark 3.5.1, p. 69]{mcduff}.

\begin{definition}(Compare McDuff and Salamon \cite[Remark 3.5.1, p. 69]{mcduff} for a similar definition)
	Let $\Sigma$ be a closed Riemann surface. Given a constant $1 < p < 2$ and a positive volume form $dvol_{\Sigma}$, we denote by $c_p (dvol_\Sigma)$ the norm of the Sobolev embedding $W^{2, p} (\Sigma) \rightarrow C^0 (\Sigma)$. That is,
	\begin{equation*}
		c_p (dvol_\Sigma) := \underset{0 \neq f \in C^\infty (\Sigma)}{\sup} \frac{\|f\|_{L^\infty}}{\|f\|_{W^{2, p}}}.
	\end{equation*}
\end{definition}

\begin{lemma}(Compare McDuff and Salamon \cite[Proposition 3.5.3, p. 70]{mcduff} for an analogous proposition)
\label{difference}
	For any $(\tilde{f}_1, \tilde{f}_2) \in \mathcal{M} (c, p)$, there exists a neighborhood $\mathcal{U}$ in $\mathcal{M} (c, p)$, and $c > 0$, and $\delta_0 = \delta_0 (c, p, \Sigma_1, \Sigma_2, x_1, x_2, N, \mathcal{U}) > 0$, such that there exists $c_0$ that depends on $c, p, \Sigma_1, \Sigma_2, x_1, x_2, N, \mathcal{U}$, and the following holds for each pair of $(\delta, R) \in \mathcal{A} (\delta_0)$ and each $(f_1, f_2) \in \mathcal{U}$. For any $\xi \in W^{2, p} (\Sigma_1 \#_{\delta, R} \Sigma_2, (f^R)^{-1} TN)$, we have
	\begin{equation*}
		\|df^R\|_{L^p} \leq c_0,\quad \|\xi\|_{L^\infty} \leq c_0,\quad c_p (dvol_{\Sigma}) \leq c_0,
	\end{equation*}
	where $\Sigma$ denotes $\Sigma_1 \#_{\delta, R} \Sigma_2$, and
	\begin{equation*}
		\| d\mathcal{F}_{f^R} (\xi) - D_{f^R}\| \leq c \|\xi\|_{W^{2, p}}.
	\end{equation*}
	Here $\|\cdot\|$ denotes the operator norm.
\end{lemma}

\begin{proof}
	By choosing small enough $\mathcal{U}$, we know that there exists $c_0 > 0$ such that the pregluing $f^R$ satisfies
	\begin{equation*}
		\|df^R\|_{L^\infty} \leq c_0
	\end{equation*}
	for each $(f_1, f_2) \in \mathcal{U}$.
	
	Denote $\Sigma_1 \#_{\delta, R} \Sigma_2$ by $\Sigma$. Given $y \in N$ and $\xi \in T_y N$, we define the (bi)linear maps
	\begin{equation*}
		E_y (\xi): T_y N \rightarrow T_{\exp_y (\xi)} N, \quad \Psi_y (\xi): T_y N \times T_y N \rightarrow T_{\exp_y (\xi)} N
	\end{equation*}
	by
	\begin{equation*}
		E_y (\xi) \xi^\prime := \frac{d}{dt} \exp_y (\xi + t\xi^\prime) \mid_{t = 0}, \quad \Psi_y (\xi; \xi^\prime, \eta) := \nabla_t (\Phi_{f^R} (\xi + t\xi^\prime)\eta)\mid_{t = 0}.
	\end{equation*}
	Now differentiate the identity
	\begin{equation*}
		\Phi_{f^R} (\xi + t\xi^\prime) \mathcal{F}_{f^R} (\xi + t\xi^\prime) = P (\exp_{f^R} (\xi + t \xi^\prime))
	\end{equation*}
	covariantly at $t = 0$ to obtain
	\begin{equation*}
		\Phi_{f^R} (\xi) d\mathcal{F}_{f^R} (\xi) \xi^\prime + \Psi_{f^R} (\xi; \xi^\prime, \mathcal{F}_{f^R} (\xi)) = D_{\exp_{f^R} (\xi)} (E_{f^R} (\xi) \xi^\prime).
	\end{equation*}
	Thus we have the following formula for $d\mathcal{F}_{f^R}$:
	\begin{equation*}
		d\mathcal{F}_{f^R} (\xi) \xi^\prime = \Phi_{f^R} (\xi)^{-1} D_{\exp_{f^R}} (\xi) (E_{f^R} (\xi) \xi^\prime) - \Phi_{f^R} (\xi)^{-1} \Psi_{f^R} (\xi; \xi^\prime, \mathcal{F}_{f^R} (\xi)).
	\end{equation*}
	Choose a constant $c_1 > 0$ such that the inequalities
	\begin{equation*}
		|E_y (\xi)| \leq c_1,\quad |\Psi_y (\xi; \xi^\prime, \eta)| \leq c_1 |\xi||\xi^\prime||\eta|
	\end{equation*}
	hold for every $y \in N$, every $\xi \in T_y N $ such that $|\xi| \leq c_0$, and every $\eta \in T_y N$. We have the pointwise estimate
	\begin{equation*}
		|\Phi_{f^R} (\xi)^{-1} \Psi_{f^R} (\xi; \xi^\prime)| \leq c_1 |d\exp_{f^R} (\xi)| |\xi| |\xi^\prime|.
	\end{equation*}
	There is a constant $c_2$, depending only on $c_0$ and the metric on $N$, such that
	\begin{equation*}
		|d\exp_{f^R} (\xi)| \leq c_2 (|du| + |\nabla \xi|).
	\end{equation*}
	Hence
	\begin{equation*}
		\|\Phi_{f^R} (\xi)^{-1} \Psi_{f^R} (\xi; \xi^\prime)\|_{L^p} \leq c_1 c_2 (\|df^R\|_{L^p} + \|\nabla \xi\|_{L^p}) \|\xi\|_{L^\infty} \|\xi^\prime\|_{L^\infty}.
	\end{equation*}
	Since $\|df^R\|_{L^p} \leq c_0$, $\|\xi\|_{L^\infty} \leq c_0$, and $\|\xi\|_{L^\infty} \leq c_0 \|\xi\|_{W^{2, p}}$ it follows that
	\begin{equation*}
		\|\Phi_{f^R} (\xi)^{-1} \Psi_{f^R} (\xi; \xi^\prime)\|_{L^p} \leq c_3 \|\xi\|_{W^{2, p}} \|\xi^\prime\|_{W^{2, p}}.
	\end{equation*}
	Now we are left to estimate
	\begin{equation*}
		\| D_{\exp_f (\xi)} (E_f (\xi) \xi^\prime) - \Phi_f (\xi) D_f \xi^\prime\|_{L^p}.
	\end{equation*}
	For this part, we only need to write out everything in coordinates and compute directly.
\end{proof}

\printbibliography
\end{sloppypar}
\end{document}